\pdfoutput=1
\documentclass[11pt,a4paper]{article}
\usepackage[a4paper, margin=3.5cm]{geometry}

\usepackage[T1]{fontenc}
\usepackage[utf8]{inputenc}
\usepackage{lmodern}

\usepackage[english]{babel}
\usepackage{amsmath}
\usepackage{bbm}

\usepackage{amsthm}
\usepackage{amssymb}
\usepackage{amsfonts}
\usepackage{dsfont}
\usepackage{mathrsfs}
\usepackage{mathtools}

\usepackage[scaled=1,sups]{XCharter}
\usepackage[charter,varbb,smallerops,scaled=1.08]{newtxmath}

\usepackage{graphicx}
\usepackage[shortlabels]{enumitem}
\usepackage{xcolor}
\usepackage[singlelinecheck=false,justification=centering]{caption}
\usepackage{hyperref}
\usepackage[all]{hypcap}

\usepackage[numbers,sort&compress]{natbib}
\usepackage{doi}

\usepackage{parskip}
\begingroup
\makeatletter
\@for\theoremstyle:=definition,remark,plain\do{%
  \expandafter\g@addto@macro\csname th@\theoremstyle\endcsname{%
    \addtolength\thm@preskip\parskip }%
}
\endgroup

\definecolor{darkblue}{rgb}{0,0,0.5}
\definecolor{darkred}{rgb}{0.5,0,0}
\definecolor{darkgreen}{rgb}{0,0.5,0}

\hypersetup{
  unicode=true,          
  pdffitwindow=false,     
  pdfstartview={FitH},    
  pdftitle={},    
  pdfauthor={},     
  pdfsubject={},   
  pdfkeywords={}, 
  colorlinks=true,       
  linkcolor=darkblue,          
  citecolor=darkred,        
  filecolor=magenta,      
  urlcolor=darkgreen           
}

\theoremstyle{plain}
\newtheorem{theorem}{Theorem}[section]
\newtheorem{corollary}[theorem]{Corollary}
\newtheorem{prop}[theorem]{Proposition}
\newtheorem{lemma}[theorem]{Lemma}

\theoremstyle{definition}
\newtheorem{definition}[theorem]{Definition}

\theoremstyle{remark}
\newtheorem{remark}[theorem]{Remark}

\newcommand{\Rplus}{[0,\infty)}
\newcommand{\R}{\mathbb{R}}
\newcommand{\N}{\mathbb{N}}
\newcommand{\Z}{\mathbb{Z}}
\newcommand{\PP}{\mathbb{P}}
\newcommand{\E}{\mathbb{E}}
\newcommand{\1}{\mathbbm{1}}
\newcommand{\eqdist}{\overset{(d)}{=}}

\newcommand{\dd}{\mathrm{d}}
\newcommand{\e}{\mathrm{e}}

\newcommand{\point}{\,\cdot\,}

\renewcommand{\epsilon}{\varepsilon}
\renewcommand{\rho}{\varrho}
\renewcommand{\phi}{\varphi}
\renewcommand{\emptyset}{\varnothing}

\newcommand{\Leb}{\ell}

\title{Mutations on a Random Binary Tree with Measured Boundary}
\author{Jean-Jil Duchamps\footnote{LPMA -UMR7599- UPMC Univ Paris 06} \footnote{CIRB -UMR7241- Collège de France} {} and Amaury Lambert\footnotemark[1] \footnotemark[2]}

\begin{document}

\maketitle

\begin{abstract}
\noindent Consider a random real tree whose leaf set, or boundary, is endowed with a finite mass measure. Each element of the tree is further given a type, or allele, inherited from the most recent atom of a random point measure (infinitely-many-allele model) on the skeleton of the tree. The partition of the boundary into distinct alleles is the so-called \emph{allelic partition}.

In this paper, we are interested in the infinite trees generated by supercritical, possibly time-inhomogeneous, binary branching processes, and in their boundary, which is the set of particles `co-existing at infinity'. We prove that any such tree can be mapped to a random, compact ultrametric tree called \emph{coalescent point process}, endowed with a `uniform' measure on its boundary which is the limit as $t\to\infty$ of the properly rescaled counting measure of the population at time $t$.

We prove that the clonal (i.e., carrying the same allele as the root) part of the boundary is a regenerative set that we characterize. We then study the allelic partition of the boundary through the measures of its blocks. We also study the dynamics of the clonal subtree, which is a Markovian increasing tree process as mutations are removed. 

\bigskip


\bigskip

\noindent MSC2000 subject classifications: primary 05C05, 60J80; secondary 54E45; 60G51; 60G55; 60G57; 60K15; 92D10.
\end{abstract}


\tableofcontents

\section{Introduction}

In this paper, we give a new flavor of an old problem of mathematical population genetics which is to characterize the so-called \emph{allelic partition} of a population. To address this problem, one needs to specify a model for the genealogy (i.e., a random tree) and a model for the mutational events (i.e., a point process on the tree). Two typical assumptions that we will adopt here are: the \emph{infinite-allele assumption}, where each mutation event confers a new type, called \emph{allele}, to its carrier; and the \emph{neutrality of mutations}, in the sense that co-existing individuals are exchangeable, regardless of the alleles they carry. Here, our goal is to study the allelic partition of the boundary of some random real trees that can be seen as the limits of properly rescaled binary branching processes. 

In a discrete tree, a natural object describing the allelic partition without labeling alleles is the \emph{allele frequency spectrum} 
$(A_k)_{k\geq 1}$, where $A_k$ is the number of alleles carried by exactly $k$ co-existing individuals in the population. In the present paper, we start from a time-inhomogeneous, supercritical binary branching process with finite population $N(t)$ at any time $t$, and we are interested in the allelic partition of individuals `co-existing at infinity' ($t\to\infty$), that is the allelic partition at the \emph{tree boundary}. To define the analogue of the frequency spectrum, we need to equip the tree boundary with a measure $\ell$, which we do as follows. Roughly speaking, if $N_u(t)$ is the number of individuals co-existing at time $t$ in the subtree $\mathcal T_u$ consisting of descendants of the same fixed individual $u$, the measure $\ell(\mathcal T_u)$ is proportional to $\lim_{t\uparrow\infty} N_u(t)/N(t)$.
It is shown in Section \ref{sec_bd_proc} that the tree boundary of any supercritical branching process endowed with the (properly rescaled) tree metric and the measure $\ell$ has the same law as a random real tree, called \emph{coalescent point process} (CPP) generated from a Poisson point process, equipped with the so-called comb metric \citep{lambert_comb_2016} and the Lebesgue measure.
Taking this result for granted, we will focus in Sections \ref{section_constr}, \ref{section_pop_clonale} and \ref{section_couplage} on coalescent point processes with mutations.

In the literature, various models of random trees and their associated allelic partitions have been considered. The most renowned result in this context is \textit{Ewens' Sampling Formula} \citep{ewens_sampling_1972}, a formula that describes explicitly the distribution of the allele frequency spectrum in a sample of $n$ co-existing individuals taken from a stationary population with genealogy given by the Moran model with population size $N$ and mutations occurring at birth with probability $\theta/N$. When time is rescaled by $N$ and $N\to\infty$, this model converges to the Kingman coalescent \citep{kingman_coalescent_1982} with Poissonian mutations occurring at rate $\theta$ along the branches of the coalescent tree. In the same vein, a wealth of recent papers has dealt with the allelic partition of a sample taken from a $\Lambda$-coalescent or a $\Xi$-coalescent with Poissonian mutations, e.g., \citep{berestycki2014, freund2012, freund2009, basdevant2008}.

In parallel, several authors have studied the allelic partition in the context of branching processes, starting with \citep{Griffiths1988} and the monograph \citep{Taib1992}, see \citep{champagnat2012birth} and the references therein. In a more recent series of papers \citep{lambert_allelic_2009, champagnat_splitting_2012, champagnat_splitting_2013, Delaporte2016}, the second author and his co-authors have studied the allelic partition at a fixed time of so-called `splitting trees', which are discrete branching trees where individuals live \textit{i.i.d} lifetimes and give birth at constant rate. In particular, they obtained the almost sure convergence of the normalized frequency spectrum $(A_k(t)/N(t))_{k\geq 1}$ as $t\to\infty$ \citep{champagnat_splitting_2012} as well as the convergence in distribution of the (properly rescaled) sizes of the most abundant alleles \citep{champagnat_splitting_2013}. The limiting spectrum of these trees is to be contrasted with the spectrum of their limit, which is the subject of the present study, as explained earlier.

Another subject of interest is the allelic partition of the entire progeny of a (sub)critical branching process, as studied in particular in \citep{bertoin_structure_2009}. The scaling limit of critical branching trees with mutations is a Brownian tree with Poissonian mutations on its skeleton. Cutting such a tree at the mutation points gives rise to a forest of trees whose distribution is investigated in the last section of \citep{bertoin_structure_2009}, and relates to cuts of Aldous' CRT in \citep{aldous_standard_1998} or the Poisson snake process \citep{abraham_poisson_2002}. 
The couple of previously cited works not only deal with the limits of allelic partitions for the whole discrete tree, but also tackle the limiting object directly. This is also the goal of the present work, but with quite different aims.


First, we construct in Section \ref{section_constr} an ultrametric tree with boundary measured by a `Lebesgue measure' $\ell$, from a Poisson point process with infinite intensity $\nu$, on which we superimpose Poissonian neutral mutations with intensity measure $\mu$. Section \ref{section_constr} ends with Proposition \ref{prop:total-nb-mut}, which states that the total number of mutations in any subtree is either finite a.s.\ or infinite a.s.\ according to an explicit criterion involving $\nu$ and $\mu$.  

The structure of the allelic partition at the boundary is studied in detail in Section \ref{section_pop_clonale}. Theorem \ref{thm_cpp_clonal} ensures that the subset of the boundary carrying no mutations (or clonal set) is a (killed) regenerative set with explicit Laplace exponent in terms of $\nu$ and $\mu$ and measure given in Corollary \ref{cor_taille_pop}. The mean intensity $\Lambda$ of the allele frequency spectrum at the boundary is defined by $\Lambda(B):=\E \sum \1_{\ell(R)\in B}$, where the sum is taken over all allelic clusters at the boundary. It is explicitly expressed in Proposition \ref{prop_allelic_freq}. An a.s.\ convergence result as the radius of the tree goes to infinity is given in Proposition \ref{prop:radius} for the properly rescaled number of alleles with measure larger than $q>0$, which is the analogue of $\sum_{k\geq q} A_{k}$ in the discrete setting.

Section \ref{section_couplage} is dedicated to the study of the dynamics of the clonal (mutation-free) subtree when mutations are added or removed through a natural coupling of mutations in the case when $\mu(\dd x)=\theta \dd x$. It is straightforward that this process is Markovian as mutations are added. As mutations are removed, the growth process of clonal trees also is Markovian, and its semigroup and generator are provided in Theorem \ref{thm_arbre_markov}.

Section \ref{sec_bd_proc} is devoted to the links between measured coalescent point processes and measured pure-birth trees which motivate the present study. Lemma \ref{lemma_link_bd_cpp+meas} gives a representation of every CPP with measured boundary, in terms of a rescaled pure-birth process with boundary measured by the rescaled counting measures at fixed times. Conversely, Theorem \ref{thm_yule} gives a representation of any such pure-birth process in terms of a CPP with intensity measure $\nu(dx) = \frac{dx}{x^2}$, as in the case of the Brownian tree.

\section{Preliminaries and Construction} \label{section_constr}

\subsection{Discrete Trees, Real Trees}
Let us recall some definitions of discrete and real trees, which will be used to define the tree given by a so-called coalescent point process.

In graph theory, a tree is an acyclic connected graph.
We call discrete trees such graphs that are labeled according to Ulam--Harris--Neveu's notation by labels in the set $\mathcal{U}$ of finite sequences of non-negative integers:
\[\mathcal{U} = \bigcup_{n\geq 0} \mathbb{Z}_+^n = \{u_1 u_2 \ldots u_n, \; u_i \in \mathbb{Z}_+, n\geq 0\},\]
with the convention $\mathbb{Z}_+^0 = \{\emptyset\}$.

\begin{definition}
A \textbf{rooted discrete tree} is a subset $\mathcal{T}$ of $\mathcal{U}$ such that
\begin{itemize}
\item $\emptyset \in \mathcal{T}$ and is called the \textbf{root} of $\mathcal T$
\item For $u = u_1 \ldots u_n \in \mathcal{T}$ and $1\leq k < n$, we have $u_1 \ldots u_k \in \mathcal{T}$.
\item For $u \in \mathcal{T}$ and $i \in \mathbb{Z}_+$ such that $ui \in \mathcal{T}$, for $0\leq j \leq i$, we have $uj\in\mathcal{T}$ and $uj$ is called a \textbf{child} of $u$.
\end{itemize}
For $n\geq 0$, the \textbf{restriction of $\mathcal{T}$ to the first $n$ generations} is defined by:
\[\mathcal{T}_{|n} := \{u \in \mathcal{T}, \; |u| \leq n\},\]
where $|u|$ denotes the length of a finite sequence.
For $u,v \in \mathcal{T}$, if there is $w\in \mathcal{U}$ such that $v = uw$, then $u$ is said to be an \textbf{ancestor} of $v$, noted $u \preceq v$.
Generally, let $u\wedge v$ denote the most recent common ancestor of $u$ and $v$, that is the longest word $u_0 \in \mathcal{T}$ such that $u_0\preceq u$ and $u_0\preceq v$.
The edges of $\mathcal{T}$ as a graph join the parents $u$ and their children $ui$. 

For a discrete tree $\mathcal{T}$, we define the \textbf{boundary of} $\mathcal{T}$ as
\[\partial\mathcal{T} := \{u\in \mathcal{T}, \; u0 \notin \mathcal{T}\}\cup\{v \in \Z_+^{\N}, \; \forall u \in\mathcal{U},u\preceq v\Rightarrow u\in \mathcal{T}\},\]
and we equip $\partial\mathcal{T}$ with the $\sigma$-field generated by the family $(B_u)_{u\in \mathcal{T}}$, where
\[B_u := \{v \in \partial\mathcal{T}, \; u\preceq v\}.\]
\end{definition}
\begin{remark} \label{rmq_justif_existence_measure_boundary}
With a fixed discrete tree $\mathcal{T}$, a finite measure $\mathscr{L}$ on $\partial\mathcal{T}$ is characterized by the values $(\mathscr{L}(B_u))_{u\in\mathcal{T}}$.
Reciprocally if the number of children of $u$ is finite for each $u\in\mathcal{T}$, by Carathéodory's extension theorem, any \emph{finitely additive} map $\mathscr{L} : \{B_u, \; u\in \mathcal{T}\} \to \Rplus$ extends uniquely into a finite measure $\mathscr{L}$ on $\partial\mathcal{T}$.
\end{remark}
%

By assigning a positive length to every edge of a discrete tree, one gets a so-called real tree.  
Real trees are defined more generally as follows, see e.g. \citep{Evans}.
\begin{definition}
A metric space $(\mathbb{T}, d)$ is a \textbf{real tree} if for all $x, y \in \mathbb{T}$,
\begin{itemize}
\item There is a unique isometry $f_{x,y} : [0,d(x,y)]\rightarrow\mathbb{T}$ such that $f_{x,y}(0) = x$ and $f_{x,y}(d(x,y)) = y$,
\item All continuous injective paths from $x$ to $y$ have the same range, equal to\\ $f_{x,y}([0,d(x,y)])$.
\end{itemize}
This unique path from $x$ to $y$ is written $[\![x,y]\!]$.
The \textbf{degree} of a point $x\in\mathbb{T}$ is defined as the number of connected components of $\mathbb{T}\setminus\{x\}$, so that we may define:
\begin{itemize}
\item The \textbf{leaves} of $\mathbb{T}$ are the points with degree $1$.
\item The \textbf{internal nodes} of $\mathbb{T}$ are the points with degree $2$.
\item The \textbf{branching points} of $\mathbb{T}$ are the points with degree larger than $2$.
\end{itemize}
One can root a real tree by distinguishing a point $\rho\in \mathbb{T}$, called the \textbf{root}.
\end{definition}

From this definition, one can see that for a rooted real tree $(\mathbb{T}, d, \rho)$, for all $x,y\in \mathbb{T}$, there exists a unique point $a\in\mathbb{T}$ such that $[\![\rho, x]\!]\cap[\![\rho, y]\!] = [\![\rho, a]\!]$.
We call $a$ the \textbf{most recent common ancestor} of $x$ and $y$, noted $x\wedge y$.
There is also an intrinsic order relation in a rooted tree: if $x\wedge y = x$, that is if $x\in [\![\rho, y]\!]$, then $x$ is called an ancestor of $y$, noted $x \preceq y$.



We will call a rooted real tree a \textbf{simple} tree
if it can be defined from a discrete tree by assigning a length to each edge.
From now on, we  will restrict our attention to simple trees.

\begin{definition} \label{def_simple_tree}
A \textbf{simple (real) tree} is given by $(\mathcal{T}, \alpha, \omega)$, where $\mathcal{T} \subset \mathcal{U}$ is a rooted discrete tree, and $\alpha$ and $\omega$ are maps from $\mathcal{T}$ to $\R$ satisfying
\[ \zeta(u) := \omega(u) - \alpha(u) > 0, \]
\[\forall u \in \mathcal{T}, \forall i \in \mathbb{Z}_+, \quad ui \in \mathcal{T} \Longrightarrow \alpha(ui) = \omega(u). \]
Here $\alpha(u)$ and $\omega(u)$ are called the \textbf{birth time} and \textbf{death time} of $u$ and $\zeta(u)$ is the \textbf{life length} of $u$.

We will sometimes consider simple trees $(\mathcal{T}, \alpha, \omega, \mathscr{L})$ equipped with $\mathscr{L}$ a \textbf{measure on their boundary} $\partial\mathcal{T}$.

We call a \textbf{reversed simple tree} a triple $(\mathcal{T}, \alpha, \omega)$ where $(\mathcal{T}, -\alpha, -\omega)$ is a simple tree.
We may sometimes omit the term ``reversed'' when the context is clear enough.

The \textbf{restriction} of $A = (\mathcal{T}, \alpha, \omega)$ to the first $n$ generations is the simple tree defined by
\[A_{|n} = (\mathcal{T}_{|n}, \alpha_{|\mathcal{T}_{|n}}, \omega_{|\mathcal{T}_{|n}}).\]
\end{definition}

One can check that a simple tree $(\mathcal{T}, \alpha, \omega)$ defines a unique real rooted tree 
defined as the completion of $(\mathbb{T}, d, \rho)$, with
\begin{equation} \label{def_simple_to_real_trees}
\begin{gathered}
\rho := (\emptyset, \alpha(\emptyset)),\\
\mathbb{T} := \{\rho\} \cup \bigcup_{u\in \mathcal{T}} \{u\}\times (\alpha(u), \omega(u)]\quad\subset \mathcal{U}\times\R,\\
d((u,x), (v,y)) := 
\begin{cases}
|x-y| & \text{if } u \preceq v \text{ or } v \preceq u, \\
x+y-2\omega(u\wedge v) & \text{otherwise}.
\end{cases}
\end{gathered}
\end{equation}
In particular, we have $(u,x)\wedge(v,y) = (u\wedge v, \omega(u\wedge v))$.

%

In this paper, we construct random simple real trees with marks along their branches.
We see these trees as genealogical/phylogenetic trees and the marks as mutations that appear in the course of evolution. We will assume that each new mutation confers a new type, called \textbf{allele}, to its bearer (infinitely-many alleles model).
Our goal is to study the properties of the \textbf{clonal subtree} (individuals who do not bear any mutations, black subtree in Figure \ref{fig_arbre_mut}) and of the \textbf{allelic partition} (the partition into bearers of distinct alleles of the population at some fixed time).

\begin{figure}[ht]
\centering
\includegraphics[width=\textwidth]{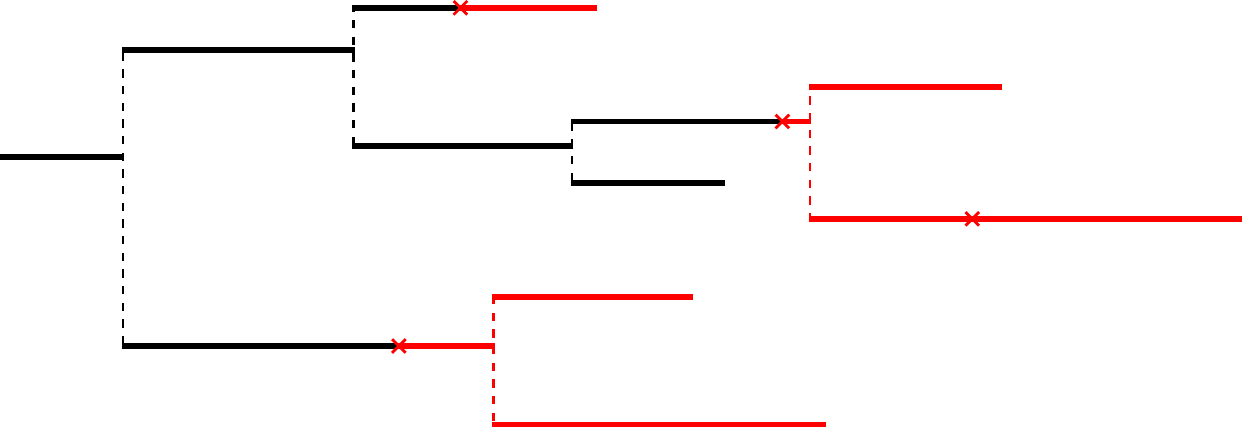}
\caption{Simple tree with mutations} \label{fig_arbre_mut}
\end{figure}

\subsection{Comb Function} \label{sect_constr_CPP}
\subsubsection{Definition}
We now introduce ultrametric trees, using a construction with comb functions following Lambert and Uribe Bravo \citep{lambert_comb_2016}.
%

\begin{definition} \label{def_comb}
Let $T>0$ and $I=[0,T]$. Let also $f:I\to \Rplus$ such that
\[ \#\{x \in I, \;f(x) > \epsilon \} < \infty \qquad \epsilon > 0.\]
The pair $(f,I)$ will be called a \textbf{comb function}. For any real number $z > \max_I f$, we define the \textbf{ultrametric tree of height $z$ associated with $(f,I)$} as the real rooted tree $T_f$ which is the completion of $(\text{Sk}, \rho, d_f)$, where $\text{Sk}\subset I\times \Rplus$ is the \textbf{skeleton} of the tree, and $\text{Sk}$, $\rho$ and $d_f$ are defined by
\begin{gather*}
\rho := (0,z),\\
\text{Sk} := \{0\}\times(0,z] \cup \{(t,y) \in I\times(0,z], \, f(t) > y\},\\
d_f : \begin{cases}
	\qquad \text{Sk}^2 & \xrightarrow{\qquad}  \Rplus \vspace*{5pt} \\ 
	\;((t,x),(s,y)) & \xmapsto{\qquad}
		{\begin{cases}
			\;|\max_{(t,s]}f - x| +  |\max_{(t,s]}f - y| \quad & \text{if } t<s,\\
			\;|x-y| & \text{if } t=s.
		\end{cases}}
\end{cases}
\end{gather*}
The set $\{0\} \times (0, z] \subset \text{Sk}$ is called the \textbf{origin branch of the tree}.
\item
For $t \in I, \, t > 0$, we call the \textbf{lineage of $t$} the subset of the tree $L_t \subset T_f$ defined as the closure of the set
\[\{(s, x) \in \text{Sk}, \; s \leq t, \, \forall s < u \leq t, \, f(u) \leq x \}. \]
For $t = 0$ one can define $L_{0}$ as the closure of the origin branch.

\end{definition}

\begin{remark}
One can check that $d_f$ is a distance which makes $(\text{Sk}, d_f)$ a real tree, and so its completion $(T_f, d_f)$ also is a real tree.
Furthermore, the fact that $\{ f > \epsilon \}$ is finite for all $\epsilon > 0$ ensures that it is a simple tree, since the branching points in $\text{Sk}$ are the points $(t, f(t))$ with $f(t)>0$.
For a visual representation of the tree associated with a comb function, see Figure \ref{fig_peigne}, where the skeleton is drawn in vertical segments and the dashed horizontal segments represent branching points.
\end{remark}

\begin{figure}[ht]
\centering
\def\svgwidth{0.5\textwidth}
\begingroup%
  \makeatletter%
  \providecommand\color[2][]{%
    \errmessage{(Inkscape) Color is used for the text in Inkscape, but the package 'color.sty' is not loaded}%
    \renewcommand\color[2][]{}%
  }%
  \providecommand\transparent[1]{%
    \errmessage{(Inkscape) Transparency is used (non-zero) for the text in Inkscape, but the package 'transparent.sty' is not loaded}%
    \renewcommand\transparent[1]{}%
  }%
  \providecommand\rotatebox[2]{#2}%
  \ifx\svgwidth\undefined%
    \setlength{\unitlength}{203.05829146bp}%
    \ifx\svgscale\undefined%
      \relax%
    \else%
      \setlength{\unitlength}{\unitlength * \real{\svgscale}}%
    \fi%
  \else%
    \setlength{\unitlength}{\svgwidth}%
  \fi%
  \global\let\svgwidth\undefined%
  \global\let\svgscale\undefined%
  \makeatother%
  \begin{picture}(1,0.96668868)%
    \put(0,0){\includegraphics[width=\unitlength,page=1]{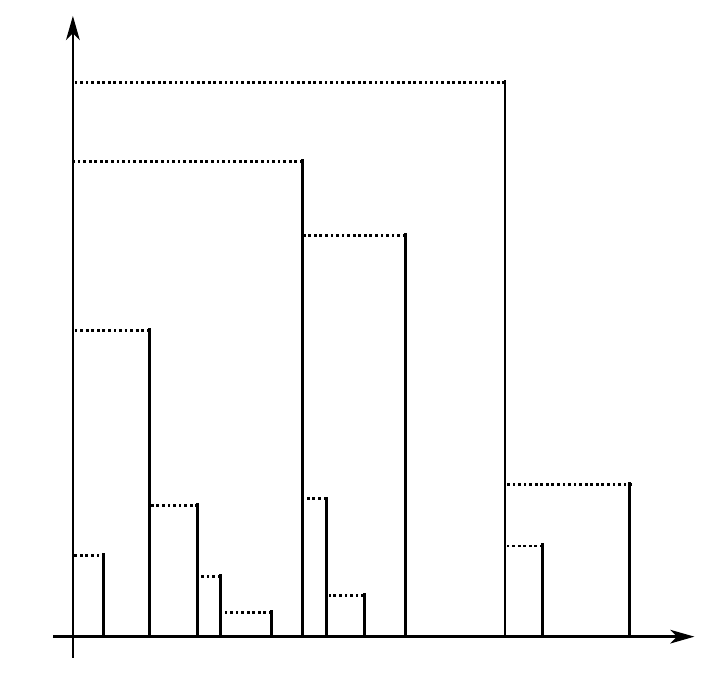}}%
    \put(0.96610999,0.00984944){\color[rgb]{0,0,0}\makebox(0,0)[lb]{\smash{$t$}}}%
    \put(0.00280956,0.92336289){\color[rgb]{0,0,0}\makebox(0,0)[lb]{\smash{$f(t)$}}}%
    \put(0.064412,0.01299976){\color[rgb]{0,0,0}\makebox(0,0)[lb]{\smash{$0$}}}%
  \end{picture}%
\endgroup%
\caption{Comb function and its associated tree.} \label{fig_peigne}
\end{figure}
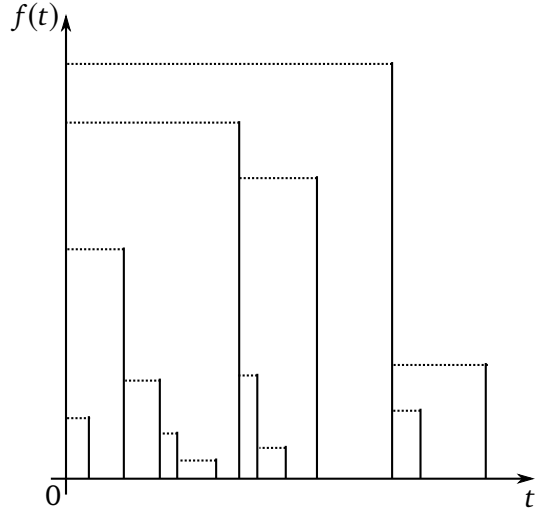

\begin{prop} \label{prop_alpha_mes}
With the same notation as in Definition \ref{def_comb}, for a fixed comb function $(f, I)$ and a real number $z> \max_I f$, writing $T_f$ for the associated real tree, the following holds.
For each $t \in I$, there is a unique leaf $\alpha_t \in T_f$ such that
\[L_t = [\![\rho, \alpha_t]\!].\]
Furthermore, the map $\alpha : t \mapsto \alpha_t$ is measurable with respect to the Borel sets of $I$ and $T_f$.
\end{prop}

\begin{proof}
For $t=0$, $L_0$ is defined as the closure of the origin branch $\{0\}\times (0,z]$.
Since $d_f((0, x), \, (0, y)) = |x-y|$, the map
\[\phi_{0} : \begin{cases}
	(0, z] &\longrightarrow  \text{Sk}  \\
	x &\longmapsto (0, x)
\end{cases} \]
is an isometry, and since $T_f$ is defined as the completion of the skeleton Sk, there is a unique isometry $\widetilde \phi_{0} : [0,z] \to T_f$ which extends $\phi_{0}$.
Therefore we define $\alpha_{0} := \widetilde \phi_{0}(0) \in T_f$, which satisfies $L_{0} = [\![\rho, \alpha_{0}]\!]$ since $\widetilde \phi_0$ is an isometry.
Also $\alpha_{0}$ is a leaf of $T_f$ because it is in $T_f \setminus\text{Sk}$.
Indeed, since $T_f$ is the completion of Sk which is connected, $T_f \setminus \{\alpha_{0} \}$ is necessarily also connected, which means that $\alpha_{0}$ has degree $1$.

Now for a fixed $t \in I, \, t>0$, write $(t_i, x_i)_{i \geq 0}$ for the (finite or infinite) sequence with values in
\[\{(0, z)\}\cup\{(s, x) \in I\times (0, \infty), \; f(s) = x \}\]
defined inductively (as long as they can be defined) by $(t_0,x_0) = (0,z)$ and
\begin{align*}
\forall i \geq 0, \quad & x_{i+1} := \max_{(t_i, t]} f ,\\
\text{and } & t_{i+1} := \max \{ s \in (t_i, t], \; f(s) = x_{i+1} \}.
\end{align*}
\begin{itemize}
\item 
If the sequence $(t_i , x_i)_{i\geq 0}$ is well defined for all $i\geq 0$, then since $f$ is a comb function, we necessarily have that $x_i \to 0$ as $i\to\infty $.
\item
On the other hand, the sequence $(t_i , x_i)_{0\leq i \leq n}$ is finite if and only if it is defined up to an index $n$ such that either $t_n = t$ or $f$ is zero on the interval $(t_n, t]$.
In that case, we still define for convenience $x_{n+i} := 0$, $t_{n+i} := t_n$ for all $i \geq 1$.
\end{itemize}
Now it can be checked that we have
\[\bigcup_{i=0}^\infty [x_{i+1}, x_i) \setminus \{0\} = (0,z),\]
and that $L_t$ is defined as the closure of the set
\[ A_t := \bigcup_{i=0}^\infty \{t_{i}\} \times \left ([x_{i+1}, x_{i})\setminus \{0\}\right ) \subset \text{Sk}.\]
Also, by definition of the sequence $(t_i , x_i)_{0\leq i}$, the distance $d_f$ satisfies, for\\$(s,x),(u,y)\in A_t$,
\[d_f((s,x),(u,y)) = |x-y|.\]
Therefore the following map is an isometry (and it is well defined because $x_i \downarrow 0$).
\[\phi_t : \begin{cases}
	(0, z) &\longrightarrow  \text{Sk} \\
	x &\longmapsto (t_i, x) \quad \text{if } x\in [x_{i+1}, x_{i}) \text{ for an index } i \geq 0.
\end{cases} \]
As in the case $t = 0$, this isometry can be extended to $\widetilde \phi_t : [0,z] \to T_f$ and we define $\alpha_t := \widetilde \phi_t(0)$.
It is a leaf of $T_f$ satisfying $L_{t} = [\![\rho, \alpha_{t}]\!]$ for the same reasons as for $0$.

It remains to prove that $\alpha : t \mapsto \alpha_t$ is measurable.
It is enough to show that it is right-continuous, because in that case the pre-image of an open set is necessarily a countable union of right-open intervals, which is a Borel set.
Now for $t < t' \in I$, by taking limits along the lineages $L_t$ and $L_{t'}$, it is easily checked that the distance between $\alpha_t$ and $\alpha_{t'}$ can be written
\[ d_f(\alpha_t, \alpha_{t'}) = 2 \max_{(t,t']} f,\]
and since $f$ is a comb function, necessarily we have 
\[\max_{(t,t']} f \xrightarrow[t' \downarrow \, t]{ } 0.\]
Hence $\alpha$ is right-continuous, therefore measurable.
\end{proof}

It follows from Proposition \ref{prop_alpha_mes} that the Lebesgue measure $\lambda$ on the real interval $I$ can be transported by the map $\alpha$ to a measure on the tree $T_f$, or more precisely on its boundary, that is the set of its leaves.

\begin{definition}
With the same notation as in Definition \ref{def_comb} and Proposition \ref{prop_alpha_mes}, for any fixed comb function $(f, I)$ and $z> \max_I f$, writing $T_f$ for the associated real tree, we define the \textbf{measure on the boundary of $T_f$} as the measure
\[\Leb:=\lambda \circ \alpha^{-1}\]
which concentrates on the leaves of the tree.
From now on, we always consider the tree $T_f$ associated with a comb function $f$ as a rooted real tree equipped with the measure $\Leb$ on its boundary.
\end{definition}

\subsubsection{The Coalescent Point Process}

Here we will consider the measured tree associated to a random comb function.
Let $\nu$ be a positive measure on $(0,\infty]$ such that for all $\epsilon > 0$, we have 
\[\overline{\nu}(\epsilon) := \nu ([\epsilon, \infty]) < \infty,
\]
 and $\mathcal{N}$ be the support of the Poisson point process on $\Rplus\times (0,\infty]$ with intensity $\dd t \otimes \nu$.
Then we can define $f^{\mathcal{N}}$ as the function whose graph 
is  $\mathcal{N}$.
\[f^\mathcal{N}(t) = \begin{cases}
x \quad & \text{ if } (t,x) \in \mathcal{N},\\
0 & \text{ if } \mathcal{N} \cap (\{t\}\times (0, \infty]) = \emptyset.
\end{cases}\]
Now fix $z>0$ such that $\overline\nu(z) > 0$ and set 
\[T(z) := \inf \{ t \geq 0, f^\mathcal{N}(t) \geq z \}.
\]

\begin{definition}
The ultrametric random tree associated to $I = [0, T(z))$ and $f^\mathcal{N}_{|I}$ is called \textbf{coalescent point process (CPP)} of intensity $\nu$ and height $z$, denoted by CPP($\nu,z$).
It is equipped with the random measure $\Leb$, concentrated on the leaves, which is the push-forward of the Lebesgue measure on $[0,T(z))$ by the map $\alpha$.
\end{definition}

Note that a coalescent point process is not directly related to coalescent theory, a canonical example of which is Kingman's coalescent \citep{kingman_coalescent_1982}, although there exist links between the two: it is shown in \citep{lambert_recovering_2016} that a CPP appears as a scaling limit of the genealogy of individuals having a very recent common ancestor in the Kingman coalescent.

Formally, a CPP is a 
random variable valued in the space of  finitely measured compact metric spaces endowed with the Gromov-Hausdorff-Prokhorov distance defined in \citep{abraham_note_2013} as an extension of the more classical Gromov-Hausdorff distance.
Actually, it is easy to check that all the random quantities we handle are measurable, since we are dealing with a construction from a Poisson point process.

%

\subsection{Mutations on a CPP} \label{sect_constr_mutations}

Here we set up how mutations appear on the random genealogy associated with a CPP of intensity $\nu$.
Let $\mu$ be a positive measure on $\Rplus$.
We make the following assumptions:
\begin{equation} \label{hyp_mu_nu} \tag{H}
\begin{gathered}
\forall x > 0, \qquad 0 < \overline{\nu}(x) := \nu ([x, \infty]) < \infty \; \text{ and } \; \underline{\mu}(x) := \mu([0,x]) < \infty, \\
{\mu}([0,\infty)) = \infty, \\
\nu \text{ and } \mu \text{ have no atom on } \Rplus.
\end{gathered}
\end{equation}
We will now define the CPP of intensity $\nu$ and height $z>0$ marked with rate $\mu$.

Recall that the CPP is constructed from the support $\mathcal{N}$ of a Poisson point process with intensity $\dd t \otimes \nu$ on $\Rplus \times [0, \infty]$ and has a root $\rho = (0,z)$.
Define independently for each point $N := (t,x)$ of $\mathcal{N}\cup\{\rho\}$ the Poisson point process $M_N$ of intensity $\mu$ on the interval $(0,x)$.
Each atom $y\in[0,x]$ of $M_N$ is a mark $(t,y)$ on the branch $\{t\}\times(0, x)\subset\text{Sk}$ at height $y$.
The family $(M_N)_{N\in\mathcal{N}}$ therefore defines a point process $M$ on the skeleton of the CPP tree:
\[M := \sum_{(t,x) \in \mathcal{N}\cup\{\rho\}} \sum_{y \in M_{(t,x)}} \delta_{(t,y)}.\]
By definition, conditional on Sk, $M$ is a Poisson point process on Sk whose intensity is such that for all non-negative real numbers $t$ and $a<b$, we have:
\[\E \Big [\, M(\{t\}\times[a,b]) \; \Big | \; \{t\}\times[a,b] \subset \text{Sk} \,\Big ] = \mu([a,b]).\]

\begin{definition} \label{def_clonal_boundary}
  Let $\nu, \mu$ be measures satisfying assumption \eqref{hyp_mu_nu}.
  A \textbf{coalescent point process with intensity $\nu$, mutation rate $\mu$ and height $z$}, denoted CPP($\nu, \mu, z$), is defined as the random CPP($\nu, z$) given by $\mathcal{N}$, equipped with the point process $M$ on its skeleton.
  \begin{enumerate}[(i)]
    \item The \textbf{clonal subtree} of the rooted real tree $(\mathbb{T}, \rho)$ equipped with mutations $M$ is defined as the subset of $\mathbb{T}$ formed by the points :
    \[ \{ x \in \mathbb{T}, M([\![\rho, x]\!]) = 0 \}.\]
    Equipped with the distance induced by $d$, this is also a real tree.
    \item Given the (ultrametric) rooted real tree $(\mathbb{T}, \rho)$ equipped with mutations $M$ and the application $\alpha$ from the real interval $I=[0,T(z))$ to $\mathbb{T}$ whose range is included in the leaves of $\mathbb{T}$, we can define the \textbf{clonal boundary} (or \textbf{clonal population}) $R = R(\mathbb{T}, M, \alpha) \subset I$:
    \[R := \{ t \in I, M([\![\rho,\alpha_t]\!]) = 0)\}.\] 
  \end{enumerate}
\end{definition}
%
%

\begin{remark}
  This set $R$ is studied in a paper by Philippe Marchal \citep{marchal_nested_2004} for a CPP with $\nu(dx) = \frac{\dd x}{x^2}$ and mutations at branching points with probability $1-\beta$.
  In that case the sets $R_\beta$ have the same distribution as the range of a $\beta$-stable subordinator.
  In the present case of Poissonian mutations, $R$ is not stable any longer but we will see in Section \ref{section_pop_clonale} that it remains a regenerative set.
\end{remark}

\paragraph{Total number of mutations.}
Since $\mu$ is a locally finite measure on $\Rplus$, the number of mutations on a fixed lineage of the CPP($\nu, \mu, z$) is a Poisson random variable with parameter $\mu([0,z])<\infty$, and so is a.s.\ finite.
However, it is possible that in a \textbf{clade} (here defined as the union of all lineages descending from a fixed point), there are infinitely many mutations with probability $1$.
For instance, if $\mu$
is the Lebesgue measure and if $\nu$ is such that 
\[\int_0 x \nu(\dd x) = \infty,\]
we know from the properties of Poisson point processes 
that the total length of any clade is a.s.\ infinite.
In this case, the number of mutations in any clade is also a.s.\ infinite so that each point $x$ in the skeleton of the tree has a.s.\ at least one descending lineage with infinitely many mutations. Such a lineage can be displayed by choosing iteratively at each branching point a sub-clade with infinitely many mutations.

One can ask under which conditions this phenomenon occurs. Conditional on the tree of height $z$, the total number of mutations follows a Poisson distribution with parameter
\[\Lambda := \underline{\mu}(z) + \sum_{(t,y)\in \mathcal{N}, t < T(z)} \underline{\mu}(y),\]
where $T(z)$ is the first time such that there is a point of $\mathcal{N}$ with height larger than $z$.
Indeed, the origin branch is of height $z$ and the heights of the other branches are the heights of points of $\mathcal{N}$.
This number of mutations is finite \textit{a.s}.\ on the event $A:=\{\Lambda<\infty\}$ and infinite \textit{a.s}.\ on its complement.
But by the properties of Poisson point processes, two cases are distinguished: either $A$ has probability $0$ or it has probability $1$.
\begin{prop}
\label{prop:total-nb-mut}
There is the following dichotomy:
  \begin{align*}
  \int_0 \underline{\mu}(x)\nu(\dd x) < \infty \quad & \Longrightarrow \quad  \text{ the total number of mutations is finite} \textit{ a.s.}\\
  \int_0 \underline{\mu}(x)\nu(\dd x) = \infty \quad & \Longrightarrow \quad \text{ the number of mutations in any clade is infinite} \textit{ a.s.}
  \end{align*}
  In the former case, the total number of mutations has mean
  \[\E[\Lambda] = \underline{\mu}(z) + \frac{1}{\overline{\nu}(z)}\int_{[0,z]} \underline{\mu}(x) \nu(\dd x). \]
\end{prop}
%
\begin{proof}
  Conditional on $T(z)$, the set $\mathcal{N}' := \{ (t, y) \in \mathcal{N}, t<T(z)\}$ is the support of a Poisson point process on $[0, T(z)]\times[0,z]$ with intensity $\dd t\otimes \nu$.
  Therefore, from basic properties of Poisson point processes, conditional on $T(z)$, $\Lambda = \underline{\mu}(z) + \sum_{(t,y)\in \mathcal{N}'} \underline{\mu}(y)$ is finite \textit{a.s}.\ if and only if
  \[ \int_{0}^{T(z)}\left ( \int_{[0,z]} \left ( \underline{\mu}(x) \wedge 1 \right ) \nu(\dd x) \right ) \dd t < \infty \qquad \textit{a.s.,}\]
  and since $T(z)$ is finite \textit{a.s}.\ and $\underline{\mu}$ is increasing, this condition is equivalent to the condition of the proposition.
  Now let us write $N_{\text{tot}}$ for the total number of mutations.
  The conditional distribution of $N_{\text{tot}}$ given $\Lambda$ is a Poisson distribution with mean $\Lambda$.
  Therefore we deduce
  \begin{align*}
   \E[N_{\text{tot}}] &= \E[\Lambda]\\
   &= \underline{\mu}(z) + \E\left [\sum_{(t,y)\in \mathcal{N}'} \underline{\mu}(y) \right ]\\
   &= \underline{\mu}(z) + \E\left [ T(z) \int_{[0,z]} \underline{\mu}(x) \nu(\dd x) \right ]\\
   &= \underline{\mu}(z) + \frac{1}{\overline{\nu}(z)}\int_{[0,z]} \underline{\mu}(x) \nu(\dd x),
  \end{align*}
  which concludes the proof.
\end{proof}

\section{Allelic Partition at the Boundary} \label{section_pop_clonale}

In this section, we will identify the clonal boundary $R$ in a mutation-equipped CPP, that is the set of leaves of the tree which do not carry mutations, and characterize the reduced subtree generated by this set.

\subsection{Regenerative Set of the Clonal Lineages, Clonal CPP}

Denote by $\mathbb{T}^z$ a CPP($\nu,\mu,z$) where $\nu,\mu$ satisfy assumptions $\eqref{hyp_mu_nu}$.
A leaf of $\mathbb{T}^z$ is said \textbf{clonal} if it carries the same allele as the root.
Recall the canonical map $\alpha^z$ from the real interval $[0, T(z))$ to the leaves of $\mathbb{T}^z$ (see Proposition \ref{prop_alpha_mes}).
The \textbf{clonal boundary} (see Definition \ref{def_clonal_boundary}) of $\mathbb{T}^z$ is then the set $R^z\subset[0, T(z))$ defined as the pre-image of the clonal leaves by the map $\alpha^z$.

We define the event
\[ O^z := \{M_{\rho}([0, z]) = 0\}\]
that there is no mutation on the origin branch of $\mathbb{T}^z$. Note that this event has a positive probability  equal to $\e^{-\underline{\mu}(z)}$.
By definition, the point process of mutations on the origin branch $M_\rho$ is independent of $(M_N)_{N\in \mathcal{N}}$.
Therefore conditioning on $O^z$ amounts to considering the tree $\mathbb{T}^z$ equipped with the mutations on its skeleton which are given only by the point processes $(M_N)_{N\in \mathcal{N}}$.
We now define a random set $\widetilde{R}$, whose distribution depends only on $(\nu, \mu)$ and not on $z$, which will allow the characterization of the clonal boundaries $R^z$ conditional on the event $O^z$.

\begin{definition} \label{def_widetilde_R} Recall the notations $\mathcal{N}$ and $(M_N)_{N\in\mathcal{N}}$. 
For each fixed $t\in\Rplus$, let $(t_i, x_i)_{i\geq 1}$ be the (possibly finite) sequence of points of $\mathcal{N}$ such that
\begin{gather*}
x_1 = \sup \{ x \in [0, \infty], \; \#\mathcal{N} \cap (0,t]\times [x, \infty] \geq 1\},\\
t_1 = \sup \{s \in [0,t], \; (s,x_1) \in \mathcal{N} \},\\
x_{i+1} = \sup \{ x \in [0, x_i), \; \#\mathcal{N}\cap(t_i,t]\times [x, \infty] \geq 1\},\\
t_{i+1} = \sup \{s \in (t_i,t], \; (s,x_{i+1}) \in \mathcal{N} \},
\end{gather*}
with the convention $\sup \emptyset = 0$, and where the sequence is finite if there is a $n\geq 0$ such that $x_n = 0$.
We define the following random point measure on $\Rplus$:
\begin{equation*} \label{eq_def_m_t}
M_t := \sum_{i\geq 1, \; x_i > 0} M_{(t_i, x_i)}(\point \cap [x_{i+1}, x_i]).
\end{equation*}
Now we define the random set $\widetilde{R}$ as:
\[\widetilde{R} := \{t \in \Rplus, \; M_t(\Rplus) = 0\}.\]

\end{definition}

\begin{remark}
Recall that for a comb function $(f,I)$ and a real number $t\in I$, in the proof of Proposition \ref{prop_alpha_mes}, we defined a sequence $(t_i, x_i)_{i\geq 0}$ in the same way as in the previous definition and we remarked that the lineage $L_t$ of $t$ is the closure of the set
\[\bigcup_{i\geq 0, \, x_i > 0} \{t_{i}\} \times \left ([x_{i+1}, x_{i})\setminus \{0\}\right ) \subset \text{Sk}.\]
It follows that in the case of the tree $\mathbb{T}^z$ equipped with the mutations $M$ on its skeleton, we have the equality between events 
\[O^z \cap \{M([\![\rho, \alpha^z_t ]\!]) = 0\} = O^z \cap \{M_t(\Rplus) = 0\}.\]
Therefore, on the event $O^z$, the clonal boundary $R^z$ of the tree $\mathbb{T}^z$ coincides with the restriction of $\widetilde{R}$ to the interval $[0, T(z))$, which explains why we study the set $\widetilde{R}$.
\end{remark}

%

The subtree of $\mathbb{T}^z$ spanned by the clonal boundary $R^z$ is called the \textbf{reduced clonal subtree} and defined as \[\bigcup_{t\in R^z} [\![\rho, \alpha^z_t]\!].\]
Note that it is a Borel subset of $\mathbb{T}^z$ because it is the closure of
\[ \bigcap_{n\geq 1} \bigcup_{p\geq n} \bigcup_{x \in C_p} [\![\rho, x]\!], \]
where $C_p$ is the finite set $\{x \in \mathbb{T}^z, \; d(x,\rho) = z(1-1/p), \, M([\![\rho, x]\!]) = 0\}$.
The set $\widetilde{R}$ is proven to be a regenerative set (see Appendix \ref{app_sub_reg} for the results used in this paper and the references concerning subordinators and regenerative sets), and the reduced clonal subtree is shown to have the law of a CPP.

\begin{theorem} \label{thm_cpp_clonal} The law of $\widetilde{R}$ and of the associated reduced clonal subtree can be characterized as follows.
    ~\begin{enumerate}[(i)] 
        \item Under the assumptions \eqref{hyp_mu_nu} and with the preceding notation the random set $\widetilde{R}$ is regenerative.
        It can be described as the range of a subordinator whose Laplace exponent $\phi$ is given by:
        \[\frac{1}{\phi(\lambda)} = \int_{(0, \infty)} \frac{\e^{-\underline{\mu}(x)}}{\lambda + \overline{\nu}(x)} \mu(\dd x).\]
        
        \item 
        The reduced clonal subtree, that is the subtree spanned by the set $\widetilde{R}$, has the distribution of a CPP with intensity $\nu^\mu$, where $\nu^\mu$ is the positive measure on $\R_+\cup\{\infty\}$ determined by the following equation.
        Letting $W(x) := (\overline{\nu}(x))^{-1}$ and $W^{\mu}(x) := (\overline{\nu^{\mu}}(x))^{-1}$, we have, for all $x>0$,
        \[ W^\mu(x) = W(0) + \int_0^x \e^{-\underline{\mu}(z)} \dd W (z). \]
    \end{enumerate}
\end{theorem}

\begin{remark}
The last formula of the theorem is an extension of Proposition 3.1 in \citep{lambert_allelic_2009}, where the case when $\nu$ is a finite measure and $\mu(dx)= \theta \,\dd x$ is treated.
Here, we allow $\nu$ to have infinite mass and $\mu$ to take a more general form (provided \eqref{hyp_mu_nu} is satisfied).
\end{remark}

\begin{figure}[ht]
\centering
\includegraphics[width=\textwidth]{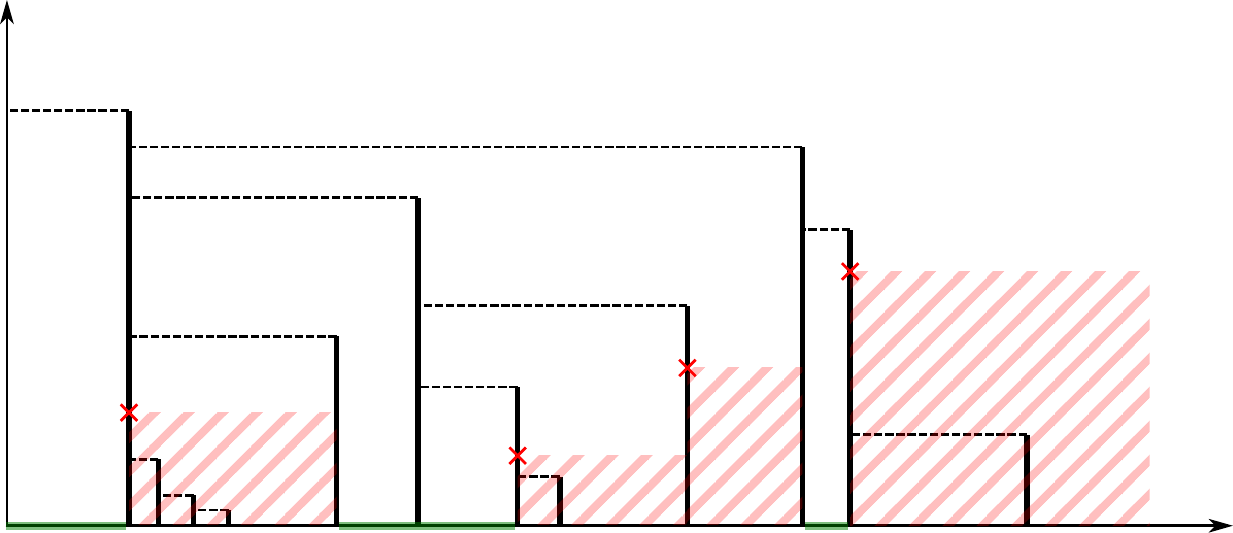}
\caption{Mutation-equipped CPP, regenerative set $\widetilde{R}$ shown in green}
\end{figure}

\paragraph{Regenerative set.}
Here, we prove the first part of the theorem concerning $\widetilde{R}$.
\begin{proof}[Proof of Theorem \ref{thm_cpp_clonal}, (i)]
Let $(\mathcal{F}_t)_{t\geq 0}$ be the natural filtration of the marked CPP defined by:
\[\mathcal{F}_t = \sigma\left (\mathcal{N}\cap([0,t]\times\R_+), M_{(s,x)}, s \leq t, x\geq 0 \right ).\]
To show first that $\widetilde{R}$ is $(\mathcal{F}_t)$-progressively measurable, we show that for a fixed $t>0$, the set
\[\{ (s, \omega)  \in [0,t]\times \Omega , \; s \in \widetilde{R}(\omega) \}\]
is in $\mathcal{B}([0,t])\otimes\mathcal{F}_t$.
Basic properties of Poisson point processes ensure there exists an $\mathcal{F}_t$-measurable sequence of random variables giving the coordinates of the mutations in $\mathcal{N}\cap([0,t]\times\R_+)$.
Let $(U_i, X_i)_i$ be such a sequence, for instance ranked such that $X_i$ is decreasing as in Figure \ref{fig_mutloc}.
We also define the following $\mathcal{F}_t$-measurable random variables:
\[T_i := t \wedge \inf\{s \geq U_i, (s,x) \in \mathcal{N}, x\geq X_i\}.\]
Now we have
\[\widetilde{R}\cap[0,t] = \bigcap_i ([0,t]\setminus[U_i, T_i)), \]
which proves that the random set $\widetilde{R}$ is $(\mathcal{F}_t)$-progressively measurable, and almost-surely left-closed.

\begin{figure}[ht]
\centering
\def\svgwidth{\textwidth}
\scriptsize
\begingroup%
  \makeatletter%
  \providecommand\color[2][]{%
    \errmessage{(Inkscape) Color is used for the text in Inkscape, but the package 'color.sty' is not loaded}%
    \renewcommand\color[2][]{}%
  }%
  \providecommand\transparent[1]{%
    \errmessage{(Inkscape) Transparency is used (non-zero) for the text in Inkscape, but the package 'transparent.sty' is not loaded}%
    \renewcommand\transparent[1]{}%
  }%
  \providecommand\rotatebox[2]{#2}%
  \ifx\svgwidth\undefined%
    \setlength{\unitlength}{364.74111963bp}%
    \ifx\svgscale\undefined%
      \relax%
    \else%
      \setlength{\unitlength}{\unitlength * \real{\svgscale}}%
    \fi%
  \else%
    \setlength{\unitlength}{\svgwidth}%
  \fi%
  \global\let\svgwidth\undefined%
  \global\let\svgscale\undefined%
  \makeatother%
  \begin{picture}(1,0.44594744)%
    \put(0,0){\includegraphics[width=\unitlength,page=1]{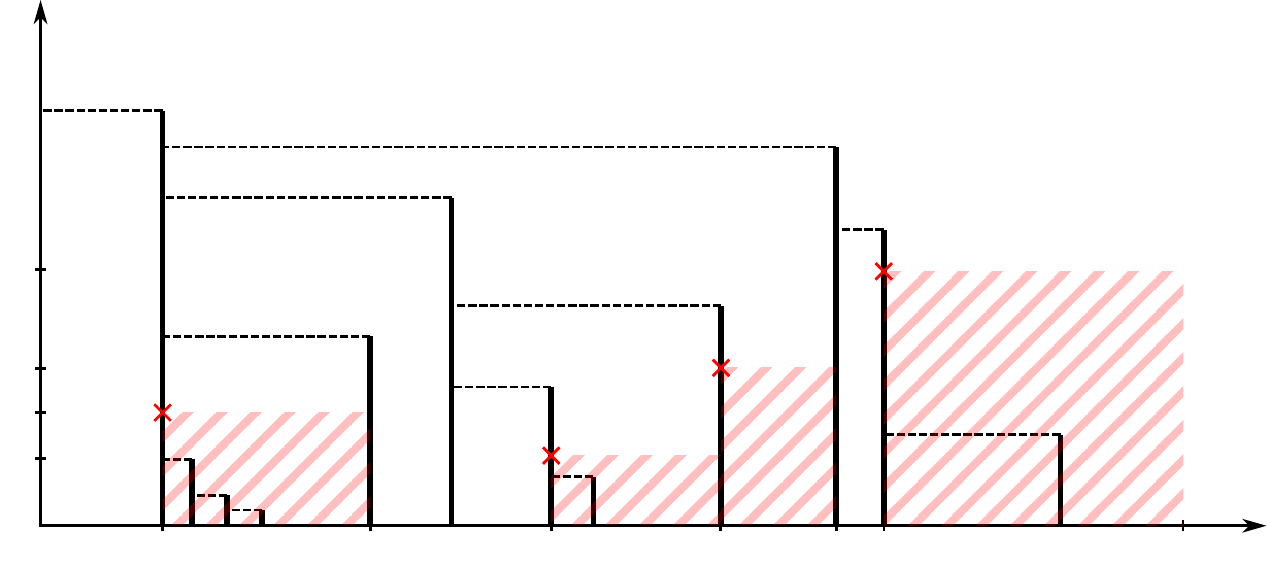}}%
    \put(-0.00182065,0.22656705){\color[rgb]{0,0,0}\makebox(0,0)[lb]{\smash{$X_1$}}}%
    \put(-0.00182065,0.07767844){\color[rgb]{0,0,0}\makebox(0,0)[lb]{\smash{$X_4$}}}%
    \put(-0.00182065,0.11334967){\color[rgb]{0,0,0}\makebox(0,0)[lb]{\smash{$X_3$}}}%
    \put(-0.00182065,0.14902088){\color[rgb]{0,0,0}\makebox(0,0)[lb]{\smash{$X_2$}}}%
    \put(0.11449859,0.00517276){\color[rgb]{0,0,0}\makebox(0,0)[lb]{\smash{$U_3$}}}%
    \put(0.68601378,0.00517276){\color[rgb]{0,0,0}\makebox(0,0)[lb]{\smash{$U_1$}}}%
    \put(0.5262687,0.00517276){\color[rgb]{0,0,0}\makebox(0,0)[lb]{\smash{$T_4=U_2$}}}%
    \put(0.42468318,0.00517276){\color[rgb]{0,0,0}\makebox(0,0)[lb]{\smash{$U_4$}}}%
    \put(0.90857124,0.00517276){\color[rgb]{0,0,0}\makebox(0,0)[lb]{\smash{$t=T_1$}}}%
    \put(0.6472407,0.00517276){\color[rgb]{0,0,0}\makebox(0,0)[lb]{\smash{$T_2$}}}%
    \put(0.27501912,0.00517276){\color[rgb]{0,0,0}\makebox(0,0)[lb]{\smash{$T_3$}}}%
  \end{picture}%
\endgroup%
\caption{Mutations localized by the variables $(U_i, X_i, T_i)$}
\label{fig_mutloc}
\end{figure}
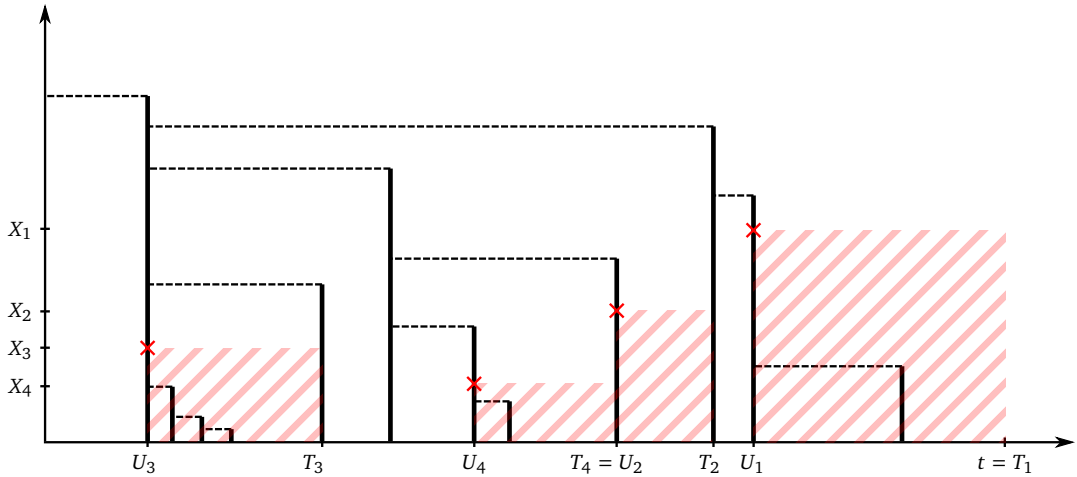

Let us now show the regeneration property of $\widetilde{R}$.
Define 
\[H(s,t) := \max\{x \geq 0, (u,x) \in \mathcal{N}, s < u \leq t \},\] the maximal height of atoms of $\mathcal{N}$ between $s$ and $t$.
We will note $H(t) := H(0,t)$ for simplicity.
Remark that
\[\widetilde{R} = \{ t \geq 0, M_t ([0,H(t)]) = 0 \}.\]
Let $S$ be a $(\mathcal{F}_t)$-stopping time, and suppose that almost surely, $S<\infty$, and $S\in \widetilde{R}$ is not isolated to the right.
From elementary properties of Poisson point processes and the fact that the random variables $(M_{(s,x)})_{s \geq 0, x \geq 0}$ are {i.i.d}, we know that the tree strictly to the right of $S$ is independent of $\mathcal{F}_S$ and has the same distribution as the initial tree.
Now since $S\in \widetilde{R}$ almost surely, we have, for all $t\geq S$,
\[M_t ([0, H(t)]) = M_t ([0,H(S,t)]),\]
because $M_t([H(S,t), H(0,t)]) = M_S([H(S,t), H(0,t)]) = 0$, in other words there are no mutations on the lineage of $t$ that is also part of the lineage of $S$.
As a consequence,
\[\widetilde{R}\cap [S, \infty) = \{ t \geq S, M_t ([0,H(S,t)]) = 0 \},\]
which implies that $\widetilde{R}\cap [S, \infty) - S$ has the same distribution as $\widetilde{R}$ and is independent of $\mathcal{F}_S$.

Therefore it is proven that $\widetilde{R}$ has the regenerative property,
%
%
 so one can compute its Laplace exponent.
Here we are in the simple case where $\widetilde{R}$ has a positive Lebesgue measure, and we have in particular, for all $t \in \R_+$,
\begin{align*}
\PP (t \in \widetilde{R}) & = \E\left [ \e^{-\underline{\mu}(H_t)} \right ]\\
& = \int_{[0, \infty]} \PP(H_t \in \dd x) \e^{-\underline{\mu}(x)}\\
& = \int_{(0, \infty)} \PP(H_t \leq x) \e^{-\underline{\mu}(x)} \mu(\dd x) \\
& = \int_{(0, \infty)} \e^{-t \overline{\nu}(x)-\underline{\mu}(x)} \mu(\dd x).
\end{align*}
The passage from the second to the third line is done integrating by parts thanks to the assumption that $\underline{\mu}$ is continuous and that $\mu$ has an infinite mass.
The last displayed expression is therefore the density with respect to the Lebesgue measure of the renewal measure of $\widetilde{R}$ (see Remark \ref{rmq_heavy_regen_set_renewal_measure}).
This is sufficient to characterize our regenerative set, and the expression given in the Proposition is found by computing the Laplace transform of this measure:
\begin{align*}
\frac{1}{\phi(\lambda)} &= \int_{0}^{\infty} \e^{-\lambda t} \left (\int_{(0, \infty)} \e^{-t \overline{\nu}(x)-\underline{\mu}(x)} \mu(\dd x)\right ) \dd t\\
&= \int_{(0, \infty)} \frac{\e^{-\underline{\mu}(x)}}{\lambda+\overline{\nu}(x)} \mu(\dd x),
\end{align*}
which concludes the proof of (i).
\end{proof}

\begin{remark}
\label{rem:popovic}
It is important to note that the particular case of a CPP with intensity $\nu(dx) = \frac{\dd x}{x^2}$ has the distribution of a (root-centered) sphere of the so-called Brownian CRT (Continuum Random Tree), the real tree whose contour is a Brownian excursion. This is shown for example by Popovic in \citep{popovic_asymptotic_2004} where the term `Continuum genealogical point process' is used to denote what is called here a coalescent point process. The measure $\nu(dx) = \frac{\dd x}{x^2}$ is the push-forward of the Brownian excursion measure by the application which maps an excursion to its depth. In general, the sphere of radius say $r$ of a totally ordered tree is an ultrametric space whose topology is characterized by the pairwise distances between `consecutive' points at distance $r$ from the root. When the order of the tree is the order associated to a contour process, these distances are the depths of the `consecutive' excursions of the contour process away from $r$, see e.g. Lambert and Uribe Bravo \citep{lambert_comb_2016}.\\
If in addition to $\nu(dx) = \frac{\dd x}{x^2}$, we assume that $\mu (dt)=\theta\, \dd t$, which amounts to letting Poissonian mutations at constant rate $\theta$ on the skeleton of the CRT, we have
\[\frac{1}{\phi_\theta(\lambda)} = \int_0^\infty \frac{\theta \e^{-\theta x}}{\lambda + 1/x} \dd x.\]
In particular, for all $\theta,c>0$, we can compute:
\[\phi_\theta(c\lambda ) = c \phi_{\theta/c}(\lambda).\]
This implies the equality in distribution $c R_\theta \overset{(d)}{=} R_{\theta / c}$.
Nevertheless $R_\theta$ is not a so-called `stable' regenerative set, contrary to the sets $R_\alpha$ in \citep{marchal_nested_2004}.
\end{remark}


\paragraph{Reduced clonal subtree.}
To show that the reduced clonal subtree is a CPP, let us exhibit the Poisson point process that generates it.
Let $\sigma$ be the subordinator with drift $1$ whose range is $\widetilde{R}$ and let $\mathcal{N}'$ be the following point process:
\[ \mathcal{N}' := \{ (t,x), \; t\in \R_+, \, x = H(\sigma_{t-},\sigma_t) > 0 \},\]
where $H(s, t) := \max\{x, (u,x) \in \mathcal{N}, s\leq u\leq t\}$.
This point process generates the reduced clonal subtree, because $H(\sigma_{t-},\sigma_t)$ is (up to a factor $1/2$) the tree distance between the consecutive leaves $\sigma_{t-}$ and $\sigma_t$ in $\widetilde{R}$.
To complete the proof of the theorem, it is sufficient to show that conditional on the death time $\zeta$ of the subordinator $\sigma$, $\mathcal{N}'$  is a Poisson point process on $[0, \zeta)\times \R_+$ with intensity $\dd t \otimes \nu^{\mu}$.

\begin{proof}[Proof of Theorem \ref{thm_cpp_clonal}, (ii)]
This is due to the regenerative property of the process.
For fixed $t\geq 0$, $\sigma_t$ is a $(\mathcal{F}_t)$-stopping time which is almost surely in $\widetilde{R}$ on the event $\{\sigma_t < \infty\} = \{\zeta > t\}$.
This implies that conditional on $\{\sigma_t < \infty\}$, the marked CPP strictly to the right of $\sigma_t$ is equal in distribution to the original marked CPP and is independent of $\mathcal{F}_{\sigma_t}$.
In particular: 
\[\left (\{(s,x) \in \R_+^2, (\sigma_t + s, x) \in \mathcal{N}\}, \widetilde{R}\cap[\sigma_t, \infty) - \sigma_t\right ) \overset{(d)}{=} (\mathcal{N}, \widetilde{R}).\]
This implies that $\mathcal{N}' \cap ([t, \infty)\times\R_+) - (t,0)$ has the same distribution as $\mathcal{N}'$ and is independent of $\mathcal{F}_{\sigma_t}$.
For fixed $\epsilon > 0$, let $(T_i, X_i)_{i\geq 1}$ be the sequence of atoms of $\mathcal{N}'$ such that $X_i > \epsilon$, ranked with increasing $T_i$.
Then $T_i$ is a $(\mathcal{F}_{\sigma_t})$-stopping time and the sequence $(T_i - T_{i-1}, X_i)_{i\geq 1}$ is i.i.d., with $T_0:=0$ for convenience.
It is sufficient to observe that $T_1$ is an exponential random variable to show that $\mathcal{N}'$ has an intensity of the form $\dd t\otimes \nu^\mu$:
\begin{align*}
\PP (T_1 > t+s \mid T_1 > t) &= \PP(H(0, \sigma_{t+s}) \leq \epsilon \mid H(0, \sigma_{t}) \leq \epsilon)\\
&= \PP(H(\sigma_t, \sigma_{t+s}) \leq \epsilon \mid H(0, \sigma_{t}) \leq \epsilon)\\
&= \PP(H(0, \sigma_{s}) \leq \epsilon) = \PP(T_1 > s).
\end{align*}
It remains to characterize the measure $\nu^\mu$ by computing $W^\mu(x)$.
Note that the following computations are correct thanks to the assumption that $\nu$ has no atom, so that $W$ is continuous.
To simplify the notation, let $H_t := H(0,t) = \max\{x, \; (u,x) \in \mathcal{N}, \, 0 \leq u \leq t\}$.
Then we can compute:
\begin{equation} \label{eq_w_mu_x}
\begin{aligned}
W^\mu(x) &= \int_0^\infty \e^{-t \overline{\nu^{\mu}}(x)} \dd t \\
&= \E \left[ \int_0^\infty \1_{\{H_{\sigma_t} \leq x\}}  \dd t\right]\\
&= \E \left [ \int_0^\infty \1_{\{H_{u} \leq x\}} \1_{\{u \in \widetilde{R}\}} \dd u \right ].
\end{aligned}
\end{equation}
Letting $F(y) := \PP (H_u \leq y) = \e^{-u \overline{\nu}(y)} $, we have
\begin{align*}\PP (H_u \leq x, u \in \widetilde{R}) &= \PP (H_u = 0) + \int_0^x \PP(H_u \in \dd y) \e^{-\underline{\mu}(y)}\\
&= F(0) + \int_0^x \e^{-\underline{\mu}(y)} \dd F (y).
\end{align*}
Now $ \dd F(y) = u \e^{-u \overline{\nu}(y)} \nu (\dd y) $, hence
\begin{align*}
W^\mu(x) &=  \int_0^\infty \e^{-u \overline{\nu}(0)} \dd u +\int_0^x \left (\int_0^\infty  u \e^{-u \overline{\nu}(y)} \dd u \right ) \e^{-\underline{\mu}(y)}\nu (\dd y)\\
&= \frac{1}{\overline{\nu}(0)} + \int_0^x \frac{1}{\overline{\nu}(y)^2} \e^{-\underline{\mu}(y)}\nu (\dd y)\\
&= W(0) + \int_0^x \e^{-\underline{\mu}(y)} \dd W(y),
\end{align*}
which concludes the proof.
\end{proof}

\begin{remark} \label{rmk_E_leb_R}
Equality \eqref{eq_w_mu_x} becomes, letting $x\to\infty$,
\[W^\mu (\infty) = \E[\lambda(\widetilde{R})].\]
\end{remark}

\begin{remark} 
In Remark \ref{rem:popovic}, we explained that when the contour of a random tree is a strong Markov process as in the case of Brownian motion, the root-centered sphere of radius $r$ of this tree is a CPP. In addition, the intensity measure of this CPP is the measure of the excursion depth under the excursion measure of the contour process (away from $r)$. 
Let $\mathbf{n}_{c}$ denote the excursion measure of the process $(B^{(c)}_t - \inf_{s\leq t}B^{(c)}_s)_{t\geq 0}$ away from $0$, with $B^{(c)}$ a Brownian motion with drift $c$, and let $h$ denote the depth of the excursion. In the case $\nu(\dd x) = \frac{\dd x}{x^2}=\mathbf{n}_{0}(h\in \dd x)$ and $\mu(\dd x) = \theta \,\dd x$, we have
\[W^\theta(x) = \frac{1-\e^{- \theta x}}{\theta} = \mathbf{n}_{\theta/2}(h\in [x, \infty])^{-1}.\]
This is consistent with Proposition 4 in \citep{abraham_poisson_2002}, which shows that putting Poissonian random cuts with rate $\theta$ along the branches of a standard Brownian CRT yields a tree whose contour process is $(e(s)-\theta s/2)_{s\geq 0}$ stopped at the first return at $0$, where $e$ is the normalized Brownian excursion. 
 \end{remark}

\subsection{Measure of the Clonal Population}
Recall that for a CPP($\nu, \mu, z$), conditional on $O^z$ (no mutation on the origin branch), the Lebesgue measure $\lambda(\widetilde{R}\cap[0,T(z))$ is equal to the measure $\ell (R^z)$ of the set of clonal leaves .

\begin{corollary} \label{cor_taille_pop}
Let $\nu, \mu$ be two measures satisfying assumptions \eqref{hyp_mu_nu}.
\begin{enumerate}[(i)]
\item With the notation of Theorem \ref{thm_cpp_clonal}, the random variable $\lambda(\widetilde{R})$ follows an exponential distribution with mean $W^\mu (\infty)$.
\item In a CPP($\nu,\mu,z$), conditional on $O^z$, the measure $\ell (R^z)$ of the set of clonal leaves is an exponential random variable of mean $W^\mu (z)$.
\end{enumerate}
\end{corollary}

\begin{proof}
Given a subordinator $\sigma$ with drift $1$ and range $\widetilde{R}$, it is known (a quick proof of this can be found in \citep{bertoin_subordinators:_1997}) that
\[\lambda(\widetilde{R}) = \inf\{ t > 0, \sigma_t = \infty \}.\]
Now the killing time of the subordinator $\sigma$ is an exponential random variable of parameter $ \phi (0) $, where $\phi$ is the Laplace exponent of $\sigma$.
We already know from Remark \ref{rmk_E_leb_R} the mean of that variable:
\[\phi(0)^{-1} = \E \left [ \lambda (\widetilde{R})\right ] = W^\mu (\infty). \]

With a fixed height $z>0$, one is interested in the law of $\lambda(\widetilde{R}\cap[0,T(z)))$. 
By the properties of Poisson point processes, stopping the CPP at $T(z)$ amounts to changing the intensity measure $\nu$ of the CPP for $\widehat{\nu}$, with
\[\widehat{\nu} = \nu(\point \cap [0,z]) + \overline{\nu}(z) \delta_\infty.\]
Then if $\widehat{W}(x) := \widehat{\nu}([x,\infty])^{-1}$, we have
\begin{align*}
\widehat{W}(x) &= \big( \nu([x,\infty]\cap[0,z]) + \overline{\nu}(z) \big )^{-1} \\
&=  \big( \nu([x\wedge z,z]) + {\nu}([z, \infty]) \big )^{-1}\\
&=  \big( \nu([x\wedge z, \infty]) \big )^{-1}\\
&= W(x\wedge z),
\end{align*}
and because of the characterization of $W^{\mu}$ given in Theorem \ref{thm_cpp_clonal}, we also have $\left (\widehat{W}\right )^{\mu}(x) = W^{\mu}(x\wedge z)$.
Therefore $\left (\widehat{W}\right )^{\mu}(\infty) = W^\mu (z)$, and we can conclude that $\lambda(\widetilde{R}\cap[0,T(z)])$ is an exponential random variable of mean $W^\mu (z)$. 
\end{proof}

\paragraph{Probability of clonal leaves.}
Here, we consider a CPP($\nu, \mu,z$) and aim at computing the probability of existence of clonal leaves in the tree.

\begin{prop} \label{prop_proba_clone}
In a CPP($\nu,\mu,z$), under the assumptions \eqref{hyp_mu_nu} and with the notation of Theorem \ref{thm_cpp_clonal}, there is a mutation-free lineage with probability
	\[\frac{W(z)\,\e^{-\underline{\mu}(z)}}{W^\mu(z)}.\]
\end{prop}

\begin{remark}
   Using a description of CPP trees in terms of birth-death trees (see Section \ref{sec_bd_proc}), the previous result could alternatively be deduced from the expression of the survival probability of a birth-death tree up to a fixed time (see Proposition \ref{prop_bd_process} in the appendix).
\end{remark}

\begin{proof}
Suppose the CPP($\nu, \mu, z$) is given by the usual construction with the Poisson point processes $\mathcal{N}$ and $(M_N)_{n\in\mathcal{N}}$.
We use the regenerative property of the process with respect to the natural filtration $(\mathcal{F}_t)_{t\geq 0}$ of the marked CPP defined by:
\[\mathcal{F}_t = \sigma\left (\mathcal{N}\cap([0,t]\times\R_+), M_{(s,x)}, s \leq t, x\geq 0 \right ).\]
Let $X$ be the first clone on the real half-line.
\[X := \inf \{x \in [0,T(z)), \, M ([\![\rho,\alpha_x]\!]) = 0\}, \]
with the convention $\inf \emptyset = \infty$ and with the usual notation.
Then $X$ is a $(\mathcal{F}_t)$-stopping time, and conditional on $\{X < \infty\}$, the law of the tree on the right of $X$ is the same as that of the original tree conditioned on having no mutation on the origin branch.
Let $C^z := \{X<\infty\}$ denote the event of existence of a mutation-free lineage. Recall that $R^z$ denotes the set of clonal leaves and that $O^z$ denotes the event that there is no mutation on the origin branch. 
Then we have
\begin{align*}
\E\left [\Leb (R^z)\right ] &= \PP(C^z) \E\left [\Leb (R^z) \mid C^z \right ]\\
&= \PP(C^z) \E\left [\Leb (R^z\cap [X, \infty) - X) \mid X < \infty \right ]\\
&= \PP(C^z) \E\left [\Leb (R^z) \mid O^z \right ]\\
&= \PP(C^z) W^\mu(z),
\end{align*}
where the last equality is due to  Corollary \ref{cor_taille_pop} (ii).
Furthermore,
\begin{align*}
\E\left [\Leb (R^z)\right ] &= \E \int_0^{T(z)} \1_{\{t\in \widetilde{R}\}} \, \dd t \\
&= \int_0^\infty \PP (t \in \widetilde{R}, t < T(z)) \, \dd t\\
&= \int_0^\infty \e^{-t\overline{\nu}(z)} \e^{-\underline{\mu}(z)} \, \dd t\\
&= \frac{\e^{-\underline{\mu}(z)}}{\overline{\nu}(z)} = W(z) \,\e^{-\underline{\mu}(z)}.
\end{align*}
Therefore, the probability that there exists a clone of the origin in the present population is
\[\PP(C^z) = \frac{W(z)\,\e^{-\underline{\mu}(z)}}{W^\mu(z)},\]
which concludes the proof.
\end{proof}

\subsection{Application to the Allele Frequency Spectrum}

\subsubsection{Intensity of the Spectrum}

From now on we fix two measures $\nu, \mu$ satisfying assumptions \eqref{hyp_mu_nu}, and we further assume for simplicity that $\overline{\nu}(z) \in (0,\infty)$ for all $z>0$. We denote by $\mathbb{T}^z$ a CPP($\nu, \mu, z$). 

 Under the infinitely-many alleles model, recall that each mutation gives rise to a new type called allele, so that the population on the boundary of the tree can be partitioned into carriers of the same allele, called \emph{allelic partition}. The key idea of this section is that expressions obtained for the clonal population of the tree allow us to gain information on quantities related to the whole allelic partition. We call $m\in \mathbb{T}^z$ a mutation if $M(\{m\})\not=0$ and denote by $\mathbb{T}^z_m$ the subtree descending from $m$. 
If $f$ is a functional of real trees (say simple, marked, equipped with a measure on the leaves), one might be interested in the quantity
\begin{equation} \label{eq_sum_mutations}
\phi(\mathbb{T}^z,f) := \sum_{\underset{\text{\footnotesize mutation}}{m \in \mathbb{T}^z}} f(\mathbb{T}^z_m),
\end{equation}
or in its expectation
\[\psi(z,f) := \E\left [\phi(\mathbb{T}^z,f)\right ].\]
For each mutation $m\in\mathbb{T}^z$, we define the set $R_m^z$ of the leaves carrying $m$ as their last mutation
\[ R_m^z := \{t \in \R_+, \; \text{the most recent mutation on the lineage of $\alpha_t^z$ is } m\}.\]
We define the random point measure putting mass on the measures of the different allelic clusters
\[\Phi_z := \sum_{\underset{\text{\footnotesize mutation}}{m \in \mathbb{T}^z}} \1_{\{R_m^z \neq \emptyset\}} \, \delta_{\lambda(R_m^z)}. \]
The intensity of the allele frequency spectrum is the mean measure $\Lambda_z$ of this point measure, that is the measure on $\R_+$ such that for every Borel set $B$ of $\R_+$,
\[\Lambda_z(B) = \E[\Phi_z(B)].\]
The analog for this measure when the number of individuals in the population is finite is the mean measure $(\E A(k))_{k > 0}$ of the number $A(k)$ of alleles carried by exactly $k$ individuals (notation $A_\theta(k,t)$ in \citep{lambert_allelic_2009} and \citep{champagnat_splitting_2012}).
The goal here is then to identify $\Lambda_z$, by noticing that for a Borel set $B$,
\[\Phi_z(B) = \phi(\mathbb{T}^z, f_B) \quad\text{and}\quad \Lambda_z(B) = \psi(z,f_B),\]
with $f_B(\mathbb{T}) := \1_{\ell(R) \in B}$, where $\mathbb{T}$ is an ultrametric tree with point mutations and measure $\ell $ supported by its leaves, and $R$ denotes the set of its clonal leaves.

\begin{prop} \label{prop_allelic_freq}
In a CPP($\nu, \mu, z$), under the assumptions \eqref{hyp_mu_nu} and with the notation of Theorem \ref{thm_cpp_clonal}, the intensity of the allele frequency spectrum has a density with respect to the Lebesgue measure:
\[\frac{\Lambda_z(\dd q)}{\dd q} = W(z)\left (\frac{\e^{-\underline{\mu}(z)}}{W^\mu(z)^2}\e^{-q/W^\mu (z)} + \int_{[0,z)}  \frac{\e^{-\underline{\mu}(x)}}{W^\mu(x)^2}\e^{-q/W^\mu (x)} \mu(\dd x)\right ). \]
\end{prop}

\begin{remark}
This expression is to be compared with Corollary 4.3 in \citep{champagnat_splitting_2012} (the term $(1-\frac{1}{W^\theta(x)})^{k-1}$ with discrete $k$ becoming here $\e^{-q/W^\mu (x)}$ with continuous $q$).
\end{remark}

\begin{remark}
Integrating this expression, we get the expectation of the number of different alleles in the population:
\[\Lambda_z(\R_+) = \E[\Phi_z(\R_+)] = W(z) \left ( \frac{\e^{-\underline{\mu}(z)}}{W^\mu(z)} + \int_{[0,z)} \frac{\e^{-\underline{\mu}(x)}}{W^\mu(x)} \mu(\dd x)\right ). \]
Note that $W(z)$ is the expectation of the total mass of the measure $\Leb$ in a CPP($\nu, \mu, z$).
It is then natural to normalize by this quantity and then let $z\to\infty$.
In \eqref{hyp_mu_nu} we assumed that $\mu([0,\infty)) = \infty$, and since $W^{\mu}(z)$ is an increasing, positive function of $z$, we have clearly $\frac{\e^{-\underline{\mu}(z)}}{W^{\mu}(z)} \to 0$ when $z\to\infty$.
Therefore we have
\[\lim_{z\rightarrow\infty} \frac{\E[\Phi_z(\R_+)]}{W(z)} = \int_{\Rplus} \frac{\e^{-\underline{\mu}(x)}}{W^\mu(x)} \mu(\dd x).\]

This provides us with a limiting spectrum intensity, written simply $\Lambda$:
\begin{equation} \label{eq_limit_spectrum}
 \frac{\Lambda(\dd q)}{\dd q} := \lim_{z\rightarrow\infty} \frac{1}{W(z)} \left (\frac{\Lambda_z(\dd q)}{\dd q}\right ) = \int_{\Rplus} \frac{\e^{-\underline{\mu}(x)}}{W^\mu(x)^2}\e^{-q/W^\mu (x)}  \mu(\dd x). 
\end{equation}
Note that in the Brownian case $\nu = \dd x / x^2$, we get a simple expression $\Lambda(\dd q) = (\theta/q) \e^{-\theta q} \dd q$.
\end{remark}

\begin{proof}[Proof of Proposition \ref{prop_allelic_freq}]
We aim at computing $\psi(z,f)$, for $f$ a measurable non-negative function of a simple real tree $\mathbb{T}$ with point mutations equipped with a measure $\ell$ on its leaves.
Suppose the mutations $(M_n)_{n\geq 1}$ on the tree $\mathbb{T}$ are numbered by increasing distances from the root.
Here we use the fact that a CPP can be seen as the genealogy of a birth-death process (see Section \ref{sec_bd_proc} for the development of this argument), a Markovian branching process whose time parameter is the distance from the root.
This description implies that, for all $n\geq 1$, conditional on the height $H_n$ of mutation $M_n$, the subtree growing from $M_n$ has the law of $\mathbb{T}^{H_n}$.
Set
\[\widetilde{f}(x) := \E[f(\mathbb{T}^x)].\]
Denoting $H_n^z$ the height of the $n$-th mutation $M_n^z$ of $\mathbb{T}^z$, we get
\begin{align*}
 \psi(z,f) &= \E \left [\sum_n f(\{\text{subtree of $\mathbb{T}^z$ growing from } M_n^z\}) \right ] \\
& = \sum_n \E\left [f(\{\text{subtree of $\mathbb{T}^z$ growing from } M_n^z\})\right ]\\
& = \sum_n \E\left [\widetilde{f}(H_n^z)\right ]\\
& = \E\left [\sum_n \widetilde{f}(H_n^z)\right ].
\end{align*}
Now this expression is simple to compute knowing $\widetilde{f}$ and the intensity of the point process giving mutation heights.
Indeed, by elementary properties of Poisson point processes
\begin{align*}
\E\left [\sum_n \widetilde{f}(H_n^z)\right ] & =
\E\left [\widetilde{f}(z) + \sum_{y \in M_{(0,z)}} \widetilde{f}(y) + \sum_{(t,x)\in \mathcal{N}, \; t \leq T(z)} \left ( \sum_{y \in M_{(t,x)}} \widetilde{f}(y) \right ) \right ]\\
& = \widetilde{f}(z) + \int_{[0,z)} \widetilde{f}(x) \mu(\dd x) + \E\left [ T(z) \int_{[0,z)} \!\!\!\nu(\dd y) \int_{[0,y)}\!\!\widetilde{f}(x)\mu(\dd x)\right ]\\
& = \widetilde{f}(z) + \int_{[0,z)} \widetilde{f}(x) \mu(\dd x) + \frac{1}{\overline{\nu}(z)} \int_{[0,z)} \widetilde{f}(x) (\overline{\nu}(x) - \overline{\nu}(z)) \mu(\dd x)\\
& = \widetilde{f}(z) + W(z) \int_{[0,z)}  \frac{\widetilde{f}(x)}{W(x)} \mu(\dd x).
\end{align*}
Now consider, for a fixed $q > 0$, the function $f$ given by $f(\mathbb{T}):= \1_{\ell(R) >q}$, where $\mathbb{T}$ is a generic ultrametric tree with point mutations and measure $\ell $ supported by its leaves, and $R$ denotes the set of its clonal leaves.
This allows us to compute the expectation $\Lambda_z((q,\infty))$ of the number of mutations carried by a population of leaves of measure greater than $q$.
Since the law of the measure of clonal leaves is known for a CPP, (see Corollary \ref{cor_taille_pop})
, we deduce
\begin{align*}
\widetilde{f}(z) &= \PP(C^z) \,\PP(\Leb(R^z) > q \mid C^z)\\
&= \PP(C^z) \,\PP(\lambda(\widetilde{R}\cap[0,T(z))) > q)\\
&=  \frac{W(z)\e^{-\underline{\mu}(z)}}{W^\mu(z)}\,\e^{-q/W^\mu (z)},
\end{align*}
where $C^z$ again denotes the event of existence of clonal leaves in $\mathbb{T}^z$ and $\widetilde{R}$ is the set defined in Definition \ref{def_widetilde_R}.
Thus we have
\begin{align*}
\Lambda_z((q,\infty)) & = \E[\Phi_z((q,\infty))] \\
& = W(z)\left (\frac{\e^{-\underline{\mu}(z)}}{W^\mu(z)}\e^{-q/W^\mu (z)} + \int_{[0,z)}  \frac{\e^{-\underline{\mu}(x)}}{W^\mu(x)}\e^{-q/W^\mu (x)} \mu(\dd x)\right ).
\end{align*}
Differentiating the last quantity yields the expression in the Proposition.
\end{proof}

\subsubsection{Convergence Results for Small Families}

Recall the construction of a CPP from a Poisson point process $\mathcal{N}$ in Section \ref{sect_constr_CPP}, and the point processes of mutations $(M_N)_{N\in\mathcal{N}}$.
Since a CPP($\nu, \mu, z$) is given by the points of $\mathcal{N}$ with first component smaller than $T(z)$, this construction yields a coupling of $(\mathbb{T}^z)_{z>0}$, where for each $z>0$, $\mathbb{T}^z$ is a CPP($\nu,\mu,z$). Recall the notation $\Phi_z$ from the previous subsection.
Then, similarly to Theorem 3.1 in \citep{lambert_allelic_2009}, we have the following almost sure convergence.
\begin{prop}
\label{prop:radius}
Under the preceding assumptions, and further assuming $\nu(\{\infty\})=0$, for any $q> 0$, we have the convergence:
\[\lim_{z\rightarrow\infty} \frac{\Phi_z((q,\infty))}{T(z)} = \int_{\Rplus}  \frac{\e^{-\underline{\mu}(x)}}{W^\mu(x)}\e^{-q/W^\mu (x)}  \mu(\dd x) = \Lambda((q,\infty)) \qquad\textit{a.s.}\]
\end{prop}
\begin{remark}
Recall that $\Phi_z((q,\infty))$ is the number of alleles carried by a population of leaves of measure larger than $q$ in the tree $\mathbb{T}^z$, and $T(z)$ is the total size of the population of $\mathbb{T}^z$.
The result is a strong law a large numbers: it shows that the number of small families (with a fixed size) grows linearly with the total measure of the tree at a constant speed given by the measure $\Lambda$ defined by \eqref{eq_limit_spectrum} as the limiting allele frequency spectrum intensity.
\end{remark}

\begin{proof}
We will use the law of large numbers several times.
Let us first introduce some notation.
For $z>0$, define $(T_i(z))_{i\ge 1}$ as the increasing sequence of first components of the atoms of $\mathcal{N}$ with second component larger than $z$, that is $T_1(z) = T(z)$ and for any $i\ge 1$
\begin{align*}
&T_{i+1}(z) = \inf \{t > T_i(z), \; \exists x > z, \, (t,x) \in \mathcal{N}\}.
\end{align*}
For $z < z'$, let $N(z,z'):=\#\{(t,x)\in \mathcal{N}: t\le T(z'), x >z\}$, that is the unique number $n$ such that
\[T_n(z) = T(z').\]
Notice that the assumptions $\overline{\nu}(z) \in (0,\infty)$ for all $z>0$ and $\nu(\{\infty\}) =0$ imply that $T(z')\to\infty$ and $N(z, z')\to\infty$  as $z'\to\infty$, for a fixed $z$.
    Because the times $(T_{i+1}(z) - T_i(z))_{i\geq 1}$ are \textit{i.i.d.} exponential random variables with mean $W(z)$ and since we have
    \[T(z') = T(z) + \sum_{i=2}^{N(z,z')} (T_{i+1}(z) - T_i(z)),\]
    it is clear by the strong law of large numbers 
    that
\[ \frac{T(z')}{N(z,z')} \underset{z' \to \infty}\longrightarrow W(z) \qquad\textit{a.s.}\]
Also, write $\mathbb{T}^z_1, \ldots, \mathbb{T}^z_{N(z,z')}$ for the sequence of subtrees of height $z$ within $\mathbb{T}^{z'}$ that are separated by the branches higher than $z$.
That is, $\mathbb{T}^z_i$ is the ultrametric tree generated by the points of $\mathcal{N}$ with first component between $T_{i-1}(z)$ and $T_i(z)$.
From basic properties of Poisson point processes, they are \textit{i.i.d.} and their distribution is that of $\mathbb{T}^{z}$.

Now, write $h(\mathbb{T})$ for the height of an ultrametric tree (i.e., the distance between the root and any of its leaves), and take any non-negative, measurable function $f$ of simple trees, such that
\begin{equation} \label{eq_prop_tronque}\tag{$\ast$}
f(\mathbb{T}) = 0 \text{ if } h(\mathbb{T}) > z.
\end{equation}
Recall the definition of $\phi(\mathbb{T}, f)$. 
Since $f$ satisfies \eqref{eq_prop_tronque}, we can write
\begin{equation} \label{eq_phi_f_repet}
 \phi(\mathbb{T}^{z'},f) = \sum_{i=1}^{N(z,z')} \phi(\mathbb{T}^z_i, f). 
\end{equation}
Therefore, again by the strong law of large numbers, we have the following convergence
\begin{equation} \label{eq_phi_f_conv}
\frac{\phi(\mathbb{T}^{z'},f)}{N(z,z')} \underset{z'\to \infty}\longrightarrow \E [\phi(\mathbb{T}^{z},f)] = \psi(z, f) \qquad \textit{a.s.}
\end{equation}
Combining the two convergence results, it follows that
\[ \frac{\phi(\mathbb{T}^{z'},f)}{T(z')} \underset{z'\to \infty}\longrightarrow \frac{\psi(z, f)}{W(z)} \qquad \textit{a.s.} \]

Let us apply this to the function $f(\mathbb{T}) = \1_{\Leb(R)>q}$.
This function $f$ does not satisfy \eqref{eq_prop_tronque} for any $z>0$, so we cannot apply \eqref{eq_phi_f_conv} directly because \eqref{eq_phi_f_repet} does not hold.
However, we can artificially truncate $f$ by defining the restriction $f^z$:
\[f^z(\mathbb{T}) := f(\mathbb{T}) \1_{h(\mathbb{T})<z},\]
which does satisfy \eqref{eq_prop_tronque}.
Now since $f^z\leq f$, we have the inequality between random variables
\[\phi(\mathbb{T}^{z'}, f^z) \leq \phi(\mathbb{T}^{z'},f),\]
and by taking limits,
\[ \frac{\psi(z, f)}{W(z)} \leq \liminf_{z'\to\infty} \frac{\phi(\mathbb{T}^{z'}, f)}{T(z')} \qquad \textit{a.s.}\]
But we have $\psi(z, f)=\Lambda_z((q,\infty))$ and as a consequence of Proposition \ref{prop_allelic_freq}, we have
\begin{align*}
\frac{\Lambda_z((q,\infty))}{W(z)} &= \frac{\e^{-\underline{\mu}(z)}}{W^\mu(z)}\e^{-q/W^\mu (z)} + \int_{[0,z)}  \frac{\e^{-\underline{\mu}(x)}}{W^\mu(x)}\e^{-q/W^\mu (x)} \mu(\dd x)\\
& \underset{z \to \infty}\longrightarrow \int_{\Rplus} \frac{\e^{-\underline{\mu}(x)}}{W^\mu(x)}\e^{-q/W^\mu (x)} \mu(\dd x),
\end{align*}
which is $\Lambda((q,\infty))$ by definition.
Therefore, we now have the inequality
\[\Lambda((q,\infty)) \leq \liminf_{z'\to\infty} \frac{\phi(\mathbb{T}^{z'}, f)}{T(z')} \qquad \textit{a.s.}\]

The converse inequality stems from a simple remark.
There are at most $N(z,z')$ mutations of height greater than $z$ giving rise to an allele carried by some leaves of $\mathbb{T}^{z'}$.
This is simply because a population of $n$ individuals can exhibit at most $n$ different alleles.
Therefore, we have
\[\phi(\mathbb{T}^{z'}, f) \leq \phi(\mathbb{T}^{z'},f^z) + N(z,z'),\]
which gives by taking limits
\[ \limsup_{z'\to\infty} \frac{\phi(\mathbb{T}^{z'}, f)}{T(z')} \leq \frac{\psi(z, f)+1}{W(z)} \underset{z\to\infty}\longrightarrow \Lambda((q,\infty)) \quad \textit{a.s.}\]
We can finally conclude
\[\frac{\phi(\mathbb{T}^{z}, f)}{T(z)} \underset{z\to\infty}\longrightarrow \Lambda((q,\infty)) \qquad \textit{a.s.},\]
which is the announced result.
\end{proof}

\section{The Clonal Tree Process} \label{section_couplage}

In this section we consider the clonal subtree $A^z$ of a random tree $\mathbb{T}^z$ with distribution CPP($\nu, \mu, z$), where $\nu, \mu$ are measures satisfying assumptions \eqref{hyp_mu_nu} and $z>0$.
We further assume $\nu(\Rplus) = \infty$, that is we ignore the case when $\mathbb{T}^z$ is a finite tree almost surely.
We will focus on the case when $\mu(dx)=\theta\,dx$.

\subsection{Clonal Tree Process}

There is a natural coupling in $\theta$ of the Poisson processes of mutations, in such a way that the sets of mutations are increasing in $\theta$ for the inclusion.
Let $\mathbb{M}$ denote a Poisson point process with Lebesgue intensity on $\R_+^2$, and define for $\theta \geq 0$,
\[\mathbb{M}^\theta := \mathbb{M}([0, \theta] \times \point).\]
Then $\mathbb{M}^\theta$ is a Poisson point process on $\R_+$ with intensity $\theta \,dx$, and the sequence of supports of  $\mathbb{M}^\theta$ increases with $\theta$. 
Let us use this idea to couple mutations with different intensities on the random tree $\mathbb{T}^z$. Recall the construction of a CPP with a Poisson point process $\mathcal{N}$ in Section \ref{section_constr}.
For each point $N = (t,x)$ of $\mathcal{N}\cup\{(0,z)\}$, let $M_N$ be a Poisson point process on $\R_+\times[0,x]$ with Lebesgue intensity.
For fixed $\theta\geq 0$, we get the original construction with $\mu (dx)= \theta \,dx$ when considering
\[M^\theta_N := M_N([0, \theta] \times \point).\]
Therefore a natural coupling of mutations of different intensities $(M^\theta)_{\theta \in \R_+}$ is defined on the random tree $\mathbb{T}^z$.
%
%
%
%
Denote $A^z_\theta$ the clonal subtree of height $z$ at mutation level $\theta$, that is the subtree of $\mathbb{T}^z$ defined by
\[A^{z}_{\theta} := \{ x \in \mathbb{T}^z, M^\theta([\![\rho, x]\!]) = 0 \}. \]
It is natural to seek to describe the decreasing process of clonal subtrees $(A^z_\theta)_{\theta\in\R_+}$.
As $\theta$ increases, it is clearly a Markov process since the distribution of $A^z_{\theta+\theta'}$ given $A^z_{\theta}$ is the law of the clonal tree obtained after adding mutations at a rate $\theta'$ along the branches of $A^z_{\theta}$.
We will now study the Markovian evolution of the time-reversed process, as $\theta$ decreases. Its transitions are relatively simple to describe using grafts of trees.

\subsection{Grafts of Real Trees}

Given two real rooted trees $(\mathbb{T}_1, d_1, \rho_1), \,(\mathbb{T}_2, d_2, \rho_2)$, and a graft point $g \in \mathbb{T}_1$, one can define the real rooted tree that is the graft of the root of $\mathbb{T}_2$ on $\mathbb{T}_1$ at point $g$ by
\[\mathbb{T}_1 \oplus_g \mathbb{T}_2 := (\mathbb{T}_1 \sqcup \mathbb{T}_2\setminus\!\{\rho_2\}, d, \rho_1),\]
with the new distance $d$ defined as follows. For any $x,y\in \mathbb{T}_1 \sqcup \mathbb{T}_2$, 
\[d(x,y) := d_i(x,y) \quad\text{ if }\quad x,y\in \mathbb{T}_i\text{ for } i \in \{1,2\},\]
and
\[d(x,y) := d_1(x,g) + d_2(\rho_2, y)\quad\text{ if }\quad x\in \mathbb{T}_1, y\in\mathbb{T}_2.\]
For real simple trees, this graft has a nice representation when the graft point is a leaf of the first tree.
\begin{definition}
For a simple tree $A = (\mathcal{T}, \alpha, \omega)$, define the \textbf{buds} of $A$ as the set $\mathcal{B}(A)$ of leaves of $\mathcal{T}$ that live a finite time
\[\mathcal{B}(A) := \{b \in \mathcal{T}, \; b0 \notin \mathcal{T}, \; \omega(b) < \infty\}.\]
For two simple trees $A_i = (\mathcal{T}_i, \alpha_i, \omega_i)$ with $i \in \{1,2\}$, and for $b \in \mathcal{B}(A_1)$, we define the \textbf{graft} of $A_2$ on $A_1$ on the bud $b$, denoted $A_1 \oplus_b A_2$ by:
\[\mathcal{T} := \mathcal{T}_1 \cup b\mathcal{T}_2,\]
\[\alpha(b) := \alpha_1(b), \quad \omega(b) := \omega_1(b)+\zeta_2(\emptyset),\]
\[\forall u \in \mathcal{T}_1\setminus\{b\}, \quad \alpha(u) := \alpha_1(u), \quad \omega(u) := \omega_1(u), \]
\[\forall u \in \mathcal{T}_2\setminus\{\emptyset\}, \quad
\begin{cases}
\alpha(bu) := \omega(b)+(\alpha_2(u)-\omega_2(\emptyset)), \\ \omega(bu) := \alpha(bu) + \zeta_2(u),
\end{cases} \]
\[A_1 \oplus_b A_2 := (\mathcal{T}, (\alpha(u), \zeta(u), \omega(u))_{u\in\mathcal{T}}).\]
It is then clear that $\mathcal{B}(A_1 \oplus_b A_2) := \mathcal{B}(A_1)\setminus\{b\}\cup b\mathcal{B}(A_2)$.
See Figure \ref{fig_greffe} for an example.
\end{definition}

\begin{figure}[ht]
\def\svgwidth{\textwidth}
\scriptsize
\begingroup%
  \makeatletter%
  \providecommand\color[2][]{%
    \errmessage{(Inkscape) Color is used for the text in Inkscape, but the package 'color.sty' is not loaded}%
    \renewcommand\color[2][]{}%
  }%
  \providecommand\transparent[1]{%
    \errmessage{(Inkscape) Transparency is used (non-zero) for the text in Inkscape, but the package 'transparent.sty' is not loaded}%
    \renewcommand\transparent[1]{}%
  }%
  \providecommand\rotatebox[2]{#2}%
  \ifx\svgwidth\undefined%
    \setlength{\unitlength}{502.77079425bp}%
    \ifx\svgscale\undefined%
      \relax%
    \else%
      \setlength{\unitlength}{\unitlength * \real{\svgscale}}%
    \fi%
  \else%
    \setlength{\unitlength}{\svgwidth}%
  \fi%
  \global\let\svgwidth\undefined%
  \global\let\svgscale\undefined%
  \makeatother%
  \begin{picture}(1,0.29011224)%
    \put(0,0){\includegraphics[width=\unitlength,page=1]{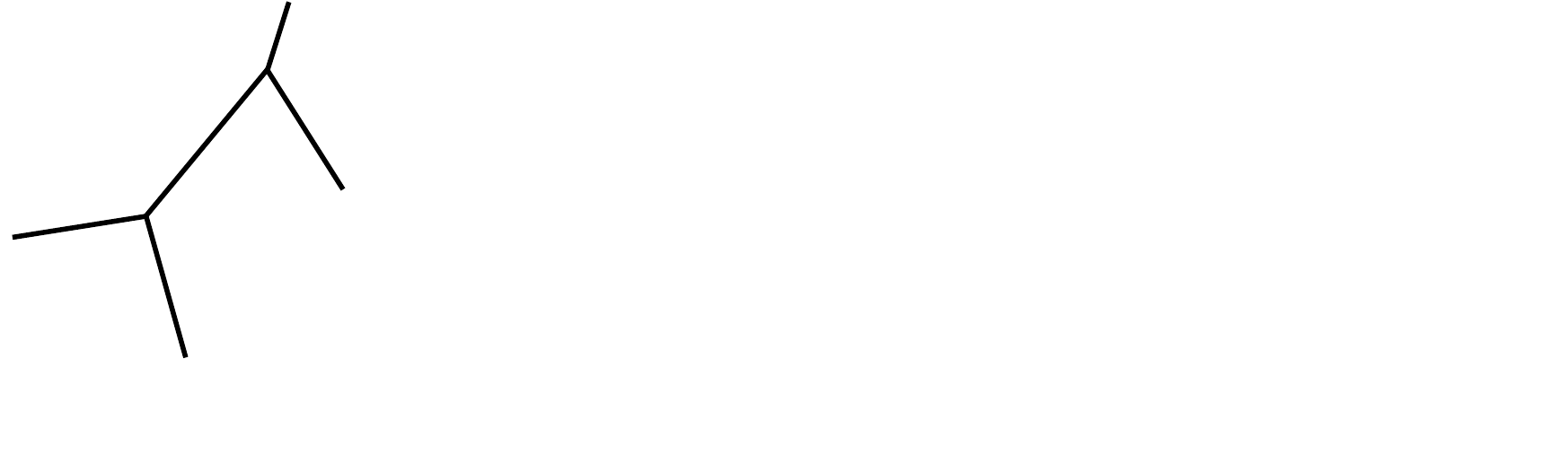}}%
    \put(0.18929089,0.24906074){\color[rgb]{0,0,0}\makebox(0,0)[lb]{\smash{$\emptyset$}}}%
    \put(0.07961298,0.16666022){\color[rgb]{0,0,0}\makebox(0,0)[lb]{\smash{$0$}}}%
    \put(0.22111456,0.1785941){\color[rgb]{0,0,0}\makebox(0,0)[lb]{\smash{$1$}}}%
    \put(-0.00165096,0.14733871){\color[rgb]{0,0,0}\makebox(0,0)[lb]{\smash{$00$}}}%
    \put(0.12564363,0.06096026){\color[rgb]{0,0,0}\makebox(0,0)[lb]{\smash{$01$}}}%
    \put(0,0){\includegraphics[width=\unitlength,page=2]{greffe.pdf}}%
    \put(0.27352028,0.12023957){\color[rgb]{0,0,0}\makebox(0,0)[lb]{\smash{$\emptyset$}}}%
    \put(0.2757934,0.0537509){\color[rgb]{0,0,0}\makebox(0,0)[lb]{\smash{$0$}}}%
    \put(0.35592077,0.13274175){\color[rgb]{0,0,0}\makebox(0,0)[lb]{\smash{$1$}}}%
    \put(0.41388529,0.08102833){\color[rgb]{0,0,0}\makebox(0,0)[lb]{\smash{$10$}}}%
    \put(0.40251971,0.17479441){\color[rgb]{0,0,0}\makebox(0,0)[lb]{\smash{$11$}}}%
    \put(0,0){\includegraphics[width=\unitlength,page=3]{greffe.pdf}}%
    \put(0.63584707,0.16571866){\color[rgb]{0,0,0}\makebox(0,0)[lb]{\smash{$0$}}}%
    \put(0.55458316,0.14639715){\color[rgb]{0,0,0}\makebox(0,0)[lb]{\smash{$00$}}}%
    \put(0.68187773,0.0600187){\color[rgb]{0,0,0}\makebox(0,0)[lb]{\smash{$01$}}}%
    \put(0,0){\includegraphics[width=\unitlength,page=4]{greffe.pdf}}%
    \put(0.79797469,0.14011307){\color[rgb]{0,0,0}\makebox(0,0)[lb]{\smash{$1$}}}%
    \put(0.80024774,0.0736244){\color[rgb]{0,0,0}\makebox(0,0)[lb]{\smash{$10$}}}%
    \put(0.86929372,0.15119456){\color[rgb]{0,0,0}\makebox(0,0)[lb]{\smash{$11$}}}%
    \put(0.9383397,0.10090183){\color[rgb]{0,0,0}\makebox(0,0)[lb]{\smash{$110$}}}%
    \put(0.92697407,0.19466793){\color[rgb]{0,0,0}\makebox(0,0)[lb]{\smash{$111$}}}%
    \put(0.0940415,0.00469084){\color[rgb]{0,0,0}\makebox(0,0)[lb]{\smash{$A_1$}}}%
    \put(0.3224898,0.00469084){\color[rgb]{0,0,0}\makebox(0,0)[lb]{\smash{$A_2$}}}%
    \put(0.73847026,0.00469084){\color[rgb]{0,0,0}\makebox(0,0)[lb]{\smash{$A_1 \oplus_1 A_2$}}}%
    \put(0,0){\includegraphics[width=\unitlength,page=5]{greffe.pdf}}%
    \put(0.74481236,0.24906074){\color[rgb]{0,0,0}\makebox(0,0)[lb]{\smash{$\emptyset$}}}%
  \end{picture}%
\endgroup%
\caption{Simple tree graft} \label{fig_greffe}
\end{figure}



\subsection{Evolution of the Clonal Tree Process}

We study the increasing clonal tree process as we remove mutations (decreasing $\theta$).
We therefore reverse time by denoting $\eta = -\ln \theta$, and defining $X^z_\eta := A^z_{\e^{-\eta}}$.
Denote $\mathbb{Q}^z_\eta$ the distribution of $X^z_\eta$ with values in the set of reversed (i.e., with time flowing from $z$ to 0) simple binary trees. 
See Figure \ref{fig_arbre_pousse} for a sketch of the tree growth process.
The increasing process $(X^z_\eta)_{\eta\in \R}$ is nicely described in terms of grafts. 
\begin{theorem} \label{thm_arbre_markov}
~\begin{enumerate}[(i)]
\item
The process $(X^z_\eta)_{\eta\in\R}$ is a time-inhomogeneous Markov  process, whose transitions conditional on $X^z_{\eta}$ can be characterized as follows.
\begin{itemize}
\item The buds of $X^z_{\eta}$ are the leaves $b$ of height $\omega(b)$.
Independently of the others, each bud $b$ is given an exponential clock $T_b$ of parameter $1$.
\item At time $\eta' = \eta + T_b$, a tree is grafted on the bud $b$, following the distribution $\mathbb{Q}^{\omega(b)}_{\eta'}$, and each newly created bud $b'$ is given an independent exponential clock $T_{b'}$ of parameter $1$.
\end{itemize}
\item
The infinitesimal generator evaluated at a function $\phi$ of simple trees which depends only on a finite number of generations ({i.e.} such that the property $\exists \, n\geq 0, \; \phi(\point) = \phi(\point_{|n})$ holds) can be written as follows
\[\mathcal{L}_\eta \phi(A) = \sum_{b \in \mathcal{B}(A)} \left ( \mathbb{Q}^{\omega(b)}_\eta  [\phi(A \oplus_b Y)] - \phi(A) \right ),\]
where $Y$ is the random tree drawn under the probability measure $\mathbb{Q}^{\omega(b)}_\eta$.
\item Write $\tau_z$ for the first time the clonal tree process reaches the boundary, that is the first time there is a leaf $x\in X^z_\eta$ with $d(\rho, x) = z$, (where $d$ is the distance in the real tree $X^z_\eta$):
\[\tau_z= \inf\{ \eta \in \R: \,\exists x\in X^z_\eta,\, d(\rho, x) = z \}.\]
Then the distribution of $\tau_z$ is given by
\[\PP(\tau_z \leq \eta) = \frac{W(z)\,\e^{-\e^{-\eta} z}}{W_\eta(z)}, \]
where as previously $W(z) = \overline{\nu}(z)^{-1}$, and
\[W_\eta (z) = W(0) + \int_{(0,z]}\e^{-\e^{-\eta} x} \dd W(x),\]
that is $W_\eta = W^\mu$ with $\mu(dx) = \e^{-\eta}\,dx$.
\end{enumerate}
\end{theorem}

\begin{figure}[ht]
\centering
\def\svgwidth{\textwidth}
\scriptsize
\begingroup%
  \makeatletter%
  \providecommand\color[2][]{%
    \errmessage{(Inkscape) Color is used for the text in Inkscape, but the package 'color.sty' is not loaded}%
    \renewcommand\color[2][]{}%
  }%
  \providecommand\transparent[1]{%
    \errmessage{(Inkscape) Transparency is used (non-zero) for the text in Inkscape, but the package 'transparent.sty' is not loaded}%
    \renewcommand\transparent[1]{}%
  }%
  \providecommand\rotatebox[2]{#2}%
  \ifx\svgwidth\undefined%
    \setlength{\unitlength}{366.78319255bp}%
    \ifx\svgscale\undefined%
      \relax%
    \else%
      \setlength{\unitlength}{\unitlength * \real{\svgscale}}%
    \fi%
  \else%
    \setlength{\unitlength}{\svgwidth}%
  \fi%
  \global\let\svgwidth\undefined%
  \global\let\svgscale\undefined%
  \makeatother%
  \begin{picture}(1,0.31283045)%
    \put(0,0){\includegraphics[width=\unitlength,page=1]{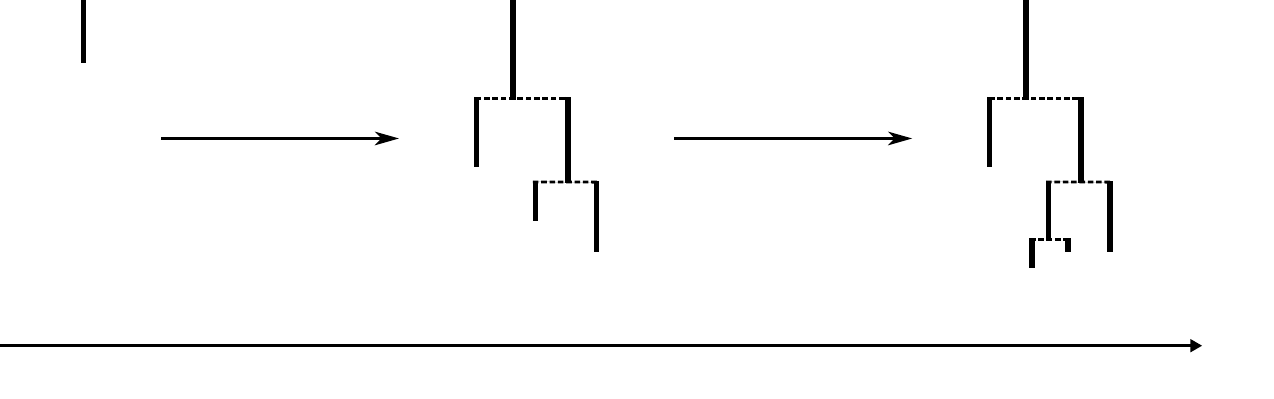}}%
    \put(0.06036935,0.24050873){\color[rgb]{0,0,0}\makebox(0,0)[lb]{\smash{$\emptyset$}}}%
    \put(0.36950378,0.16117994){\color[rgb]{0,0,0}\makebox(0,0)[lb]{\smash{$1$}}}%
    \put(0.40855569,0.12019154){\color[rgb]{0,0,0}\makebox(0,0)[lb]{\smash{$21$}}}%
    \put(0.45856827,0.09608857){\color[rgb]{0,0,0}\makebox(0,0)[lb]{\smash{$22$}}}%
    \put(0.12545105,0.17583167){\color[rgb]{0,0,0}\makebox(0,0)[lb]{\smash{$\emptyset$ grows at time $\eta_2$}}}%
    \put(0.79086195,0.08328368){\color[rgb]{0,0,0}\makebox(0,0)[lb]{\smash{$211$}}}%
    \put(0.82391723,0.0969243){\color[rgb]{0,0,0}\makebox(0,0)[lb]{\smash{$212$}}}%
    \put(0.05243774,0.00899966){\color[rgb]{0,0,0}\makebox(0,0)[lb]{\smash{$\eta_1$}}}%
    \put(0.77172159,0.16117994){\color[rgb]{0,0,0}\makebox(0,0)[lb]{\smash{$1$}}}%
    \put(0.86748448,0.09746646){\color[rgb]{0,0,0}\makebox(0,0)[lb]{\smash{$22$}}}%
    \put(0,0){\includegraphics[width=\unitlength,page=2]{arbre_pousse.pdf}}%
    \put(0.40716407,0.00899966){\color[rgb]{0,0,0}\makebox(0,0)[lb]{\smash{$\eta_2$}}}%
    \put(0.81124357,0.00899966){\color[rgb]{0,0,0}\makebox(0,0)[lb]{\smash{$\eta_3$}}}%
    \put(0.51807968,0.17583167){\color[rgb]{0,0,0}\makebox(0,0)[lb]{\smash{$21$ grows at time $\eta_3$}}}%
  \end{picture}%
\endgroup%
\caption{Markovian evolution of an increasing tree process. In this example, the time $\eta_2 - \eta_1$ is an exponential time of parameter $1$ and $\eta_3 - \eta_2$ is an exponential time of parameter $3$.}
\label{fig_arbre_pousse}
\end{figure}
We first state a result that is already interesting in itself, which ensures that CPP trees are reversed pure-birth trees (see next Section for details on birth-death trees and their links with CPPs). We refer the reader to Subsection  \ref{subsec:proof_lemma_link_bd_cpp}, where a more general result is proved.
\begin{lemma} \label{lemma_link_bd_cpp}
  Let $\nu$ and $\mu$ be diffuse measures on $\Rplus$, satisfying assumptions \eqref{hyp_mu_nu} and $\nu(\Rplus)=\infty$.
  Fix $z_0\in \Rplus$ such that $\overline{\nu}(z_0) =1$ and let $J = (0, z_0]$.
  Then for $z\in J$, a CPP$(\nu, z)$ is the genealogy of a reversed ({i.e.} with time flowing from $z$ to $0$) pure-birth process with birth intensity $\beta$ defined as the Laplace-Stieltjes measure associated with the nondecreasing function  $ -\log \overline{\nu}$, started from $z$.
\end{lemma}

\begin{proof}[Proof of Theorem \ref{thm_arbre_markov}]
From Lemma \ref{lemma_link_bd_cpp}, we can express the CPP in terms of a pure-birth tree, with time flowing from $z$ to $0$ (but measured from 0 to $z$)  and birth intensity $\dd \beta = \dd (\log \circ W)$.
Let $\mathcal{T} \subset \mathcal{U}$ denote the complete binary tree
\[\mathcal{T} := \bigcup_{n\geq 0}\{0, 1\}^n\]
Then we can define recursively $(\alpha(u), \omega(u))_{u\in \mathcal{T}}$ by setting $\alpha(\emptyset) = z$, and for $u = vi$, with $i\in \{0,1\}$:
\[\alpha(u) = \omega(v) = \sup [0, \alpha(v)) \cap B_v,\]
with the convention $\sup \emptyset = 0$, and where $(B_v)_{v\in\mathcal{T}}$ are \textit{i.i.d.} Poisson point processes on $[0,z]$ with intensity $\beta$.
This defines the random reversed simple tree $(\mathcal{T}, \alpha, \omega)$ as the genealogy of a pure-birth process with birth intensity $\beta$, with time flowing from $z$ to $0$.
In other words, by the definition of $\beta$, $(\mathcal{T}, \alpha, \omega)$ is the reversed simple tree with distribution CPP($\nu, z$).

Now we define independently of $(\mathcal{T}, \alpha, \omega)$, a family $({M}_u)_{u\in \mathcal{T}}$ of \textit{i.i.d.} Poisson point processes on $\Rplus\times[0,z]$ with Lebesgue intensity.
Writing for $\eta \in \R$ and $u\in\mathcal{T}$,
\[{M}^\eta_u = {M}_u([0,\e^{-\eta}]\times \point),\]
we define a coupling $((M_u^\eta)_{u \in \mathcal{T}})_{\eta\in \R}$ of point processes with intensity $\e^{-\eta} \dd x$ on the branches of $(\mathcal{T}, \alpha, \omega)$.

Now let us define the process $(Y_\eta)_{\eta\in \R}$ by $Y_\eta = (\mathcal{T}_\eta, \alpha_\eta, \omega_\eta)$, with
\begin{gather*}
 \mathcal{T}_\eta := \{u\in \mathcal{T}, \; \forall v \prec u, \, M^\eta_v([\alpha(v), \omega(v)]) = 0 \},\\
 \alpha_\eta(u) := \alpha(u) \quad \forall u \in \mathcal{T}_\eta,\\
 \text{and }\; \omega_\eta (u) := \sup \left (\{\omega(u)\}\cup\{ s < \alpha(u), \; M_u^\eta([s, \alpha(u)]) = 0 \}\right ) \quad\forall u \in \mathcal{T}_\eta,
\end{gather*}

By definition, one can check that $Y_\eta$ is the clonal simple tree associated with the tree $(\mathcal{T}, \alpha, \omega)$ and the point process of mutations $(M_u^\eta)_{u \in \mathcal{T}}$.
Therefore $(Y_\eta)_{\eta\in\R}$ has the same distribution as $(X^z_\eta)_{\eta\in\R}$.
We define the filtration $(\mathcal{F}_\eta)_{\eta\in \R}$ as the natural filtration of the process $(Y_\eta)_{\eta\in\R}$, which we may rewrite:
\[ \mathcal{F}_\eta := \sigma\left ( (\alpha_{\eta'})_{\eta'\leq \eta}, (\omega_{\eta'})_{ \eta'\leq \eta} \right ). \]
From our definitions, for $u\in \mathcal{T}$, we have:
\[\omega_\eta(u) = \inf \{s \in [0, \alpha(u)],\; M_{u}([0,\e^{-\eta}]\times[s,\alpha(u)]) = 0 \text{ and } B_{u}([s,\alpha(u)]) = 0 \},\]
and since $M_{u}$ and $B_u$ are independent Poisson point processes, it is known that conditional on $\mathcal{F}_\eta$, we have: $M_u\cap \Rplus \times [0, \omega_\eta(u))$ and $B_u \cap [0, \omega_\eta(u))$ are independent Poisson point processes, with intensity Lebesgue for $M_u$ and $\beta$ for $B_u$, on their respective domains.

We can further notice that on the event $\{u \text{ is a bud of } Y_\eta\}$, conditional on $\mathcal{F}_\eta$, the families of point processes
\[ (M_{uv}\cap \Rplus \times [0, \omega_\eta(u)))_{v \in \mathcal{T}} \text{ and } (B_{uv}\cap [0, \omega_\eta(u)))_{v \in \mathcal{T}} \]
are independent families of independent Poisson point process with intensity Lebesgue for $M_{uv}$ and $\beta$ for $B_{uv}$, on their respective domains.

Also, since $M_u$ and $B_u$ are independent and with diffuse intensities, we have the a.s.\ equalities between events
\begin{align*}
&\{u \text{ is a bud of } Y_\eta\} \\
&\quad = \{\omega_\eta(u) = \inf \{s \in [0, \alpha(u)],\; M_{u}([0,\e^{-\eta}]\times[s,\alpha(u)]) = 0\} \}\\
&\quad =\{B_u( \{\omega_\eta(u)\}) = 0 \}.
\end{align*}
Moreover, since $M_{u}$ is a Poisson point process with Lebesgue intensity on $\Rplus^2$, it is known that on this event, conditional on $\omega_\eta(u)$, the point process $M_{u}$ restricted to $[0, \e^{-\eta}]\times[0,\omega_\eta(u)]$ has the conditional distribution of:
\[\delta_{(U, \, \omega_\eta(u))} + \widehat{M},\]
where $U$ is a uniform random variable on $[0, \e^{-\eta}]$ and $\widehat{M}$ is an independent Poisson point process on $[0, \e^{-\eta}]\times[0,\omega_\eta(u))$ with Lebesgue intensity.
Hence on the event $A := \{u \text{ is a bud of } Y_\eta\}$, the distribution of 
\begin{align*}
    \widehat{\eta} &= \inf \{\eta' \geq \eta, \; M_u([0, \e^{-\eta'}]\times \{\omega_\eta(u)\}) = 0\}\\
    &= \sup \{\eta' \geq \eta, \; M_u([0, \e^{-\eta'}]\times \{\omega_\eta(u)\}) = 1\}
\end{align*}
is given by
\begin{align*}
\PP(\widehat{\eta}-\eta \geq t \mid A ) & = \PP (M_u([0, \e^{-(\eta+t)}]\times \{\omega_\eta(u)\}) =1 \mid A)\\
& = \PP (U \in [0, \e^{-(\eta+t)}])\\
&= \e^{-t},
\end{align*}
And so if $u$ is a bud of $Y_\eta$, the first time $\widehat \eta$ such that $\omega_{\widehat{\eta}}(u)$ is lower than $\omega_\eta(u)$ satisfies that $\widehat{\eta} - \eta$ has an exponential distribution with parameter $1$.

We may now prove the first point \textit{(i)} of the theorem.
Fix $\eta\in\R$, and write $(b_1, b_2, \ldots)$ for the distinct buds of $Y_\eta$.
We define, for $i \geq 1$ and $\eta' \geq \eta$:
\begin{gather*}
\widetilde{\mathcal{T}}^i_{\eta'} := \{u, \; b_i u \in \mathcal{T}_{\eta'} \},\\
\widetilde {\alpha}^i_{\eta'}(\emptyset) := \omega_\eta(b_i) \text{ and for }u \in \mathcal{T}\setminus \{\emptyset\},\; \widetilde {\alpha}^i_{\eta'}(u) := \alpha_{\eta'}(b_i u), \\
 \widetilde {\omega}^i_{\eta'}(u) := \omega_{\eta'}(b_i u), \\
\widetilde {Y}^i _{\eta'} := \left (\widetilde{\mathcal{T}}^i_{\eta'}, \widetilde {\alpha}^i_{\eta'}, \widetilde {\omega}^i_{\eta'}\right ).
\end{gather*}
This definition formulates that for $\eta' \geq \eta$, $\widetilde {Y}^i_{\eta'}$ is the unique simple tree such that $Y_{\eta'} = A \oplus_{b_i} \widetilde {Y}^i_{\eta'}$ for another simple tree $A$ in which $b_i$ is a bud, with $\omega^A(b_i) = \omega_\eta(b_i)$.
Note that when writing $Y_{\eta'} = A \oplus_{b_i} \widetilde {Y}^i_{\eta'}$, $A$ may be different from $Y_\eta$, even for $\eta'$ arbitrarily close to $\eta$, since other grafts may have occurred (possibly infinitely many grafts if $Y_\eta$ has infinitely many buds).

Since $b_1,b_2, \ldots$ are the buds of $Y_\eta$, the sets $b_1 \mathcal{T},b_2 \mathcal{T},\ldots$ are disjoint.
Thus, from our construction, the following families of random variables are independent conditional on $\mathcal{F}_\eta$:
\[(B_{b_1 u})_{u\in \mathcal{T}},(B_{b_2 u})_{u\in \mathcal{T}} \ldots, (M_{b_1 u})_{u\in \mathcal{T}}, (M_{b_2 u})_{u\in \mathcal{T}}, \ldots\]
Furthermore, we know how to describe their distributions conditional on $\mathcal{F}_\eta$ because of the previous observations.
It follows that the trees $(\widetilde {Y}^i_{\eta'})_{i\geq 1}$ are independent conditional on $\mathcal{F}_\eta$ and the distribution of $(\widetilde {Y}^i_{\eta'})_{\eta' \geq \eta}$ can be described by:

There is a random variable $\widehat{\eta}$ such that 
\begin{itemize}
\item $\widehat{\eta} - \eta$ is exponentially distributed with parameter $1$.
\item For $\eta \leq \eta' < \widehat{\eta}$, we have $\omega_{\eta'}(b_i) = \omega_\eta(b_i)$ so $\widetilde {Y}^i_{\eta'}$ is the empty tree (or rather contains only one point, the root).
\item Conditionally on $\widehat{\eta}$, the process $(\widetilde {Y}^i_{\eta'})_{\eta' \geq \widehat{\eta}}$ is distributed as our construction of the process $(Y_{\eta'})_{\eta'\geq \widehat{\eta}}$, with the initial condition $\alpha(\emptyset) = {\omega_\eta(b_i)}$.
\end{itemize}
This concludes the proof of \textit{(i)}.

For \textit{(ii)}, write $\mathfrak{T}$ for the set of simple binary trees and suppose we have a bounded measurable map $\phi : \mathfrak{T} \to \R$ and a number $n\geq 0$ such that
\[\phi(A) = \phi(A_{|n}) \qquad A \in \mathfrak{T}. \]
Consider a fixed tree $A = (\mathcal{T}, \alpha, \omega) \in \mathfrak{T}$.
There is a finite number of buds $b_1, \ldots, b_m$ in the first $n$ generations $\mathcal{T}_{|n}$, therefore for a fixed $\eta \in \R$, conditional on  $\{X^z_{\eta} = A\}$, the process $(\phi(X^z_{\eta'}))_{\eta' \geq \eta}$ is a continuous time Markov chain.
It follows from \textit{(i)} that this Markov chain jumps after an exponential time with parameter $m$ to a new state where one of the buds, uniformly chosen, grows into a new tree.
That is, denoting $\mathcal{L}_\eta$ the infinitesimal generator of the process $(X^z_\eta)_{\eta \geq \eta_0}$,
\[\mathcal{L}_\eta \phi(A) = \sum_{i=1}^{m} \left ( \mathbb{Q}^{\omega(b_i)}_\eta  [\phi(A \oplus_{b_i} Y)] - \phi(A) \right ),\]
where $Y$ is the random tree drawn under the probability measure $\mathbb{Q}^{\omega(b_i)}_\eta$.

For \textit{(iii)}, note that the existence of a leaf in the clonal subtree at a distance $z$ from the root coincides a.s.\ with the existence of a clonal leaf in $\mathbb{T}^z$, where $\mathbb{T}^z$ is the original CPP($\nu,z$) with mutation measure $\mu(\dd x) = \e^{-\eta}\, \dd x$.
Then the formula in the proof follows from Proposition \ref{prop_proba_clone}, which gives the probability that there is a clonal leaf in a CPP.
\end{proof}

\paragraph{The branching random walk of the buds.}
Forgetting the structure of the tree and considering only the height of the buds, the process becomes a rather simple branching random walk.
Write $\chi^z_\eta := \sum_{b\in \mathcal{B}(X^z_\eta)} \delta_{\omega(b)}$ for the point measure on $\R_+$ giving the heights of the buds in $X^z_\eta$.
Then $(\chi^z_\eta)_{\eta\geq \eta_0}$ is a branching Markov process where each particle stays at their height $z'$ during their lifetime (an exponential time of parameter $1$), then splits at their death time $\eta$ according to the distribution of $\chi^{z'}_\eta$.
Similarly to the preceding paragraph, one can describe the infinitesimal generator of this process as follows.
For a map $f:\R_+ \to \R_+$ that is zero in a neighborhood of $0$ and a Radon point measure $\Gamma$ on $(0,\infty)$ ({i.e.} such that $\Gamma([x,\infty)) < \infty \; \forall x >0$), write $\phi^f(\Gamma)$ for the sum
\[\phi^f(\Gamma) := \int f(z) \Gamma( \dd z). \]
Then the infinitesimal generator $\mathcal{L}_\eta$ at time $\eta$ of the time-inhomogeneous process $(\chi_\eta)_{\eta\in\R}$, evaluated at $\phi^f$, is given by
\[\mathcal{L}_\eta \phi^f(\Gamma) = \int \mathbb{Q}^z_\eta [\phi^f(\chi)] \,\Gamma(\dd z) - \phi^f(\Gamma). \]

\section{Link between CPP and Birth-Death Trees} \label{sec_bd_proc}

\subsection{Birth-Death Processes} 
An additional well-known example of random tree is given by the genealogy of a birth-death process, which will appear as an alternative description of our CPP trees.
Here, a birth-death process is a time-inhomogeneous, time-continuous Markovian branching process living in $\Z_+$ with jumps in $\{-1, 1\}$.
In a general context, we will define the genealogy of a birth-death process as a random simple tree, which we may equip with a canonical limiting measure on the set of its infinite lineages.

Let $J = [t_0, t_\infty)$ be a real interval, with $-\infty < t_0 <t_\infty\leq\infty$.
Suppose there are two measures on $J$, $\beta$ and $\kappa$, respectively called the birth intensity measure and death intensity measure, or simply birth rate and death rate, which satisfy for all $t\in J$
\begin{equation} \label{eq_diffuse_radon_measures}
\begin{gathered}
\beta([t_0, t])<\infty, \qquad \kappa([t_0, t]) < \infty\\
 \beta(\{t\}) = 0, \qquad \kappa(\{t\}) = 0.
\end{gathered}
\end{equation}
In other words, $\beta$ and $\kappa$ are diffuse Radon measures on $J$.

Informally, the population starts with one individual at time $t_0$, and each individual alive at time $t\geq t_0$ may give birth to a new individual at rate $\beta(\dd t)$, and die at rate $\kappa(\dd t)$.

\begin{definition} \label{def_bd_process}
  Let $J = [t_0, t_\infty)$ be a real interval, with $-\infty < t_0 <t_\infty\leq\infty$, and $\beta$ and $\kappa$ measures on $J$ satisfying \eqref{eq_diffuse_radon_measures}.
  Independently for each $u\in \bigcup_{n}\{0,1\}^n$, we define $B_u$ and $D_u$ two independent point processes, such that $B_u$ ({resp.} $D_u$) is a Poisson point process on $J$ with intensity $\beta$ ({resp.} $\kappa$).
  
  The \textbf{genealogy of a $(\beta, \kappa)$ birth-death process} started from $t \in J$ is the random binary simple tree $(\mathcal{T}, \alpha, \omega)$ defined recursively by:
  \begin{enumerate}
    \item $\emptyset \in \mathcal{T}$, with $\alpha(\emptyset) = t$.
    \item For each $u\in \mathcal{T}$, we set $T_B(u) := \inf B_u\cap (\alpha(u),t_\infty)$, and $T_D(u) :=\inf D_u\cap (\alpha(u),t_\infty)$.
    Then there are three different possibilities:
    \begin{itemize}
      \item if $T_B(u) < T_D(u)$, then we set $u0, u1 \in \mathcal{T}$, and $\alpha(u0) = \alpha(u1) = \omega(u) := T_B(u)$,
      \item if $T_D(u) < T_B(u)$, then we set $\omega(u) = T_D(u)$, and $u0, u1 \notin \mathcal{T}$,
      \item if $T_B(u)=T_D(u)=t_\infty$, then we set $\omega(u) = t_\infty$, and $u0, u1 \notin \mathcal{T}$.
    \end{itemize}
  \end{enumerate}
\end{definition}
Birth-death processes have been known for a long time. They have been studied thoroughly as early as 1948 \citep{kendall_generalized_1948}.
In the case of pure-birth processes with infinite descendance, we introduce a canonical measure on the boundary of the tree.

\begin{definition} \label{def_bd_proc_measure_on_boundary}
  Under the assumption $\kappa = 0$ and $\beta(J)=\infty$, the tree $(\mathcal{T}, \alpha, \omega)$ is said to be the genealogy of a \textbf{pure-birth process}. It may then be equipped with a \textbf{measure $\mathscr{L}$ on its boundary $\partial\mathcal{T}= \{0,1\}^{\N}$} defined by
  \[
   \mathscr{L}\left (B_u\right ) := \lim _{s\uparrow t_\infty} \frac{N_u(s)}{\e^{\beta([t_0,s])}} \qquad  u\in \mathcal{T},
  \]
  where $B_u =\{v\in \partial\mathcal{T}, u\prec v\}$ is defined as in Definition \ref{def_simple_tree}, and $N_u(s)$ is the number of descendants of $u$ at time $s$: 
  \[N_u(s) := \#\{v\in \mathcal{T}, \; u\preceq v, \, \alpha(v) < s \leq \omega(v)\}.\]
\end{definition}

\begin{remark}
  The limits in the definition are well-defined because for each $u\in\mathcal{T}$, conditional on $\alpha(u)$, the process $\left (\frac{N_u(s)}{\e^{\beta([t_0,s])}}\right )_{s \geq \alpha(u)}$ is a non-negative martingale.
  Also, the fact that the map $u\mapsto N_u(s)$ is additive combined with Remark \ref{rmq_justif_existence_measure_boundary} justifies that the measure $\mathscr L$ is well defined.
\end{remark}

Finally, let us introduce random mutations on a birth-death tree as a random discrete set of points.
\begin{definition}
\label{def_bd_process+mut}
  Let $\mu$ be a diffuse Radon measure on $J$, and let $\#$ denote the counting measure on
  $\bigcup_{n}\{0,1\}^n$.
  A birth-death tree $(\mathcal{T}, \alpha, \omega)$ may be equipped with a set $M$ of \textbf{neutral mutations} at rate $\mu$ by defining, independently of the preceding construction, a Poisson point process $\widetilde{M}$ on $(\bigcup_{n}\{0,1\}^n)\times J$ with intensity $\#\otimes\mu$, and then defining:
  \[M := \{(u,s) \in \widetilde{M},\; \,u\in \mathcal{T}, \alpha(u) < s \leq \omega(u)\}.\]
  This point process $M$ is then a discrete subset of the skeleton of the real tree (defined as in \eqref{def_simple_to_real_trees}) associated with $(\mathcal{T}, \alpha, \omega)$.
\end{definition}  

\textbf{Example.} 
The Yule tree is the genealogy of a pure-birth process with $J=\Rplus$ and a birth rate $\beta$ equal to the Lebesgue measure, which means that the branches separating two branching points are \textit{i.i.d} exponential random variables with parameter $1$.
Every pure-birth tree with $\beta(J)=\infty$ can be time-changed into a Yule tree, with the time-change $\phi:J\to \Rplus, \, t\mapsto \beta([t_0, t])$ (see Proposition \ref{prop_bd_time_change}).


\subsection{Link between CPP and Supercritical Birth-Death Trees} \label{sec_yule}
We first provide a refined version of Lemma \ref{lemma_link_bd_cpp} which is proved in Subsection \ref{subsec:proof_lemma_link_bd_cpp}.
\begin{lemma} \label{lemma_link_bd_cpp+meas}
  Under the assumptions of Lemma \ref{lemma_link_bd_cpp}, the CPP$(\nu, z)$ with \textbf{boundary measured by $\ell$} is the genealogy of a reversed pure-birth process with birth intensity $\dd\beta= -\dd\log \overline{\nu}$ started from $z$, \textbf{with boundary measured by $\mathscr{L}$}.
  \end{lemma}

Let $J = [t_0, t_\infty)$ be a real interval, with $-\infty < t_0 <t_\infty\leq\infty$, and let $\beta$ and $\kappa$ be diffuse Radon measures on $J$, {i.e.} measures satisfying \eqref{eq_diffuse_radon_measures}.
Consider a birth-death process started from $t_0$ with birth rate $\beta$ and death rate $\kappa$.
Let us define
\begin{gather*}
\mathcal{I}_t := \int_{[t, t_\infty)} \e^{-\beta([t,s])+\kappa([t,s])}\, \beta(\dd s)\\
\beta^*(\dd t) := \frac{\beta(\dd t)}{\mathcal{I}_t}
\end{gather*}
In a birth-death process with $\beta(J) = \infty$, we say that an individual $i$ alive at time $t$ has an \emph{infinite progeny} if $N_i(s)>0$ for any time $s>t$.
It is known (see \citep{kendall_generalized_1948}) that the process is supercritical (i.e., the event $\{\liminf_{t\to t_\infty} N_\emptyset(t) >0 \}$ has positive probability) if and only if $\mathcal{I}_{t_0} < \infty$, and that the probability of non-extinction for a process started at time $t\in J$ is then $\mathcal{I}_t^{-1}$. Also, if the birth-death process with rates $(\beta, \kappa)$ is supercritical, then conditional on non-extinction, the subtree of individuals with infinite progeny is a pure-birth tree with birth rate $\beta^*$.

Now we assume Poissonian neutral mutations are set on the genealogy of a $(\beta, \kappa)$ supercritical birth-death process, according to a rate $\mu$, where $\mu$ is a diffuse Radon measure on $J$.
We also assume $\beta^*(J) =\infty$ so that $\lim_{t\to t_\infty} N_\emptyset(t) =+\infty$ conditional on non-extinction.
Conditional on non-extinction, the subtree of individuals with infinite progeny is a measured simple tree equipped with mutations $(\mathcal{T}, \alpha, \omega, \mathscr{L}, M)$, where:
\begin{itemize}
\item $(\mathcal{T}, \alpha, \omega, \mathscr{L})$ is a random simple binary tree constructed (see Definition \ref{def_bd_proc_measure_on_boundary}) from a pure-birth process with birth rate $\beta^*$.
\item With $\widehat M$ a Poisson point process on $(\bigcup_{n\geq 0}\{0,1\}^n)\times J$ with intensity $\#\otimes \mu$, the mutations on the branches of $\mathcal{T}$ are defined as the set
\[M = \{(i,t)\in \widehat{M},\,i\in \mathcal{T}, \alpha(i) < t \leq \omega(i)\}\]
\end{itemize}
One may study this measured tree with mutations as the limit in time of the genealogy of the birth-death process with neutral mutations.
We show that this measured tree with mutations is in fact a time-changed CPP tree.

\begin{theorem} \label{thm_yule}
  Let $J = [t_0, t_\infty)$ be a real interval, with $-\infty < t_0 <t_\infty\leq\infty$, and let $\beta$ and $\mu$ be diffuse Radon measures on $J$, with $\beta(J) =\infty$.
  Let $\mathbb{T}=(\mathcal{T}, \alpha, \omega, M, \mathscr{L})$ be a random measured simple tree representing the genealogy of a pure-birth process with rate $\beta$ started from $t_0$, equipped with mutations at rate $\mu$.
  Let $\phi: J \to (0,1]$ be the time-change defined by
  \[\phi : t \mapsto \e^{-\beta([t_0, t))}.\]
  Then the time-changed tree $\phi(\mathbb{T})$ (see Proposition \ref{prop_bd_time_change}) has the distribution of a\\ CPP$\left (\frac{\dd x}{x^2},\mu \circ\phi^{-1},1\right)$.
\end{theorem}

\begin{proof}
Thanks to Lemma \ref{lemma_link_bd_cpp+meas}, we only need to exhibit a correct time change to prove the Theorem.
  We know that a time-changed birth-death tree is still a birth-death tree: this is explicitly stated in Proposition \ref{prop_bd_time_change} in the appendix.
  This implies here that the time-changed tree $\phi(\mathbb{T})$ is a (reversed) pure-birth process with birth rate $\beta\circ\phi^{-1}$, started from $\phi(t_0)=1$, and equipped with mutations with rate $\mu\circ\phi^{-1}$.
  Let us first check that $\beta\circ\phi^{-1}(\dd x) = \dd \log (x)$.
  Since $\beta$ is diffuse, $\phi$ is continuous decreasing, so for all $x\in (0,z_0]$, we have $\phi(\phi^{-1}(x))=x$, where $\phi^{-1}$ is the right-continuous inverse of $\phi$.
  Therefore we have, for all $a<b \in (0,1]$:
  \begin{align*}
  \beta\circ\phi^{-1}([a,b]) &= \beta([\phi^{-1}(b), \phi^{-1}(a)]) \\
  &= \log \phi(\phi^{-1}(b)) - \log \phi(\phi^{-1}(a))\\
  &= \log(b) - \log(a).
  \end{align*}
  Now notice that for $x\in (0,1]$,
  \[-\log\left (\int_x^\infty \frac{1}{y^2}\dd y\right) = \log x, \]
  so according to Lemma \ref{lemma_link_bd_cpp+meas}, a CPP$\left (\frac{\dd x}{x^2},\mu \circ\phi^{-1},1\right)$ is a pure-birth process with birth rate $\beta(\dd x)= \dd \log (x)$, started from $1$ and equipped with mutations at rate $\mu \circ\phi^{-1}$.
  Therefore its distribution is identical to the distribution of $\phi(\mathbb{T})$.
\end{proof}

\paragraph{Acknowledgements.} The authors thank the \emph{Center for Interdisciplinary Research in Biology} (Coll\`ege de France) for funding.

\appendix
\addtocontents{toc}{\protect\setcounter{tocdepth}{2}}

\section{Appendix}

\subsection{Birth-Death Processes}

\begin{prop} \label{prop_bd_process}
Let $J = [t_0, t_\infty)$ be a real interval, with $-\infty < t_0 <t_\infty\leq\infty$, and $\beta$ and $\kappa$ diffuse Radon measures on $J$ ({i.e.} satisfying \eqref{eq_diffuse_radon_measures}).
Let $\PP_t$ denote the distribution of the genealogy of a $(\beta, \kappa)$ birth-death process started with one individual at time $t\in J$, and let $N_T$ be the number of individuals alive at time $T \in J$.
For $T>t$ and $\alpha\geq 0$, we have:
\[\E_t(\e^{-\alpha N_T}) = 1-\frac{(1-\e^{-\alpha})}{\e^{\kappa([t,T]) - \beta([t,T])}+(1-\e^{-\alpha})\int_{[t, T]} \e^{\kappa([t,s]) - \beta([t,s])} \beta(\dd s)},\]
and in particular,
\[\PP_t(N_T > 0) = \left (\e^{\kappa([t,T]) - \beta([t,T])}+\int_{[t, T]} \e^{\kappa([t,s]) - \beta([t,s])} \beta(\dd s)\right )^{-1}. \]
\end{prop}

  \begin{remark}
  Note that the previous proposition shows that conditional on being non-zero, $N_T$ is a geometric random variable, which is a known fact about birth-death processes (see for instance \citep{kendall_generalized_1948}).
  We still provide a proof in our case where the birth and death intensity measures are not necessarily absolutely continuous with respect to Lebesgue.
  \end{remark}
  
\begin{proof}
With a fixed time horizon $T\in J$ and a fixed real number $\alpha \geq 0$, write for $t < T$,
\[q(t) = \E_t(\e^{-\alpha N_T}).\]
We use a different description of the birth-death process than the one used in Section \ref{sec_bd_proc}, and consider a population where individuals die at rate $\kappa$, and during their lifetime, produce a new individual at rate $\beta$.
Notice that for any $s>t$, the number of individuals alive at time $s$ has the same distribution in both models.

Thus we write $D$ for the death time of the first individual, and $B_i$ for the possible birth time of her $i$-th child.
With our description, $D$ has the distribution of the first atom of a Poisson point process on $[t, t_\infty)$ with intensity $\kappa$ and conditional on $D$, the set $\{B_1, B_2, \ldots, B_N\}$ is a Poisson point process on $[t, D]$ with intensity $\beta$.
Also, write $\widetilde{N}^i_T$ for the number of alive descendants of the $i$-th child at time $T$.
Since we have $N_T = \1_{D>T} + \sum_{i} \widetilde{N}^i_T$, we have
\[q(t) = \E_t\left [ \e^{-\alpha \1_{D>T}} \textstyle \prod_{i}\e^{-\alpha \widetilde{N}^i_T}  \right ],\]
where we define by convention $\widetilde{N}^i_T = 0$ if $B_i > T$.
Now conditional on $D$ and $(B_i)$, $(\widetilde{N}^i_T)$ are independent, with $\widetilde{N}^i_T$ equal to the distribution of $N_T$ under $\PP_{B_i}$.
Hence
\[q(t) = \E_t\left [ \e^{-\alpha \1_{D>T}} \textstyle \prod_{i}q(B_i)  \right ],\]
where we use the convention $q(u) := 1$ if $u > T$.
Now conditional on $D$, $(B_i)$ are the atoms of a Poisson point process with intensity $\beta(\dd s)$ on $[t, D]$, so we have
\begin{align*}
q(t) &= \E_t\left [ \e^{-\alpha \1_{D>T}} \exp\left ({-\int_{[t,D]} (1-q(s))\beta(\dd s)}\right ) \right ]\\
 &= \int_{[t, \infty)} \kappa(\dd u) \, \e^{-\kappa([t, u))} \e^{-\alpha \1_{u>T}} \exp\left ({-\int_{[t,u]} (1-q(s))\beta(\dd s)}\right ),
\end{align*}
which implies by differentiation
\[\dd q (t) = - \kappa(\dd t)+q(t)\left [\kappa(\dd t) + (1-q(t)) \beta(\dd t)\right ],\]
which in turn may be rewritten
\[\dd \left (\frac{1}{1-q(t)}\right ) = -\beta(\dd t) + \left (\frac{1}{1-q(t)}\right ) (\beta(\dd t)-\kappa(\dd t)). \]
Remark that with $F(t) := \e^{\beta([t,T]) - \kappa([t,T])}$, we have $\dd F(t) = F(t)(\kappa(\dd t) - \beta(\dd t))$, so that
\[\dd \left (\frac{F(t)}{1-q(t)}\right ) = -F(t) \beta(\dd t),\]
and since $q(T) = \e^{-\alpha}$, we have by integration on $[t, T]$:
\[\frac{1}{1-\e^{-\alpha}} - \frac{F(t)}{1-q(t)} = -\int_{[t, T]} F(s) \beta(\dd s),\]
that is
\[ 1- q(t) = \frac{(1-\e^{-\alpha})}{\e^{\kappa([t,T]) - \beta([t,T])}+(1-\e^{-\alpha})\int_{[t, T]} \e^{\kappa([t,s]) - \beta([t,s])} \beta(\dd s)}. \]
This characterizes the distribution of $N_T$ under $\PP_t$ for all $T$.
In particular, letting $\alpha\to\infty$, we get
\[\PP_t(N_T > 0) = \left (\e^{\kappa([t,T]) - \beta([t,T])}+\int_{[t, T]} \e^{\kappa([t,s]) - \beta([t,s])} \beta(\dd s)\right )^{-1},\]
which concludes the proof.
\end{proof}

\begin{prop}[Time-changed birth-death processes] \label{prop_bd_time_change}
Let $J = [t_0, t_\infty)$ be a real interval, with $-\infty < t_0 <t_\infty\leq\infty$, and $\beta$, $\kappa$, and $\mu$ diffuse Radon measures on $J$ ({i.e.} satisfying \eqref{eq_diffuse_radon_measures}).
Let $\phi:J\to \R$ be an increasing function, and define $t'_0 := \phi(t_0)$, $t'_\infty := \lim_{t\uparrow t_\infty}\phi(t)$ and $J'=[t'_0, t'_\infty)$.
We assume that $\phi$ satisfies 
\[\forall t<t_\infty,\;\phi(t)<t'_\infty.\]

Let $\mathbb{T}=(\mathcal{T}, \alpha, \omega, M)$ be the genealogy of a $(\beta, \kappa)$ birth-death process, started at $t\in J$ and equipped with Poissonian mutations with rate $\mu$, as in Definition \ref{def_bd_process+mut}.
We define the time-changed simple tree:
\[\phi(\mathbb{T}) := (\mathcal{T}, \phi \circ \alpha, \phi \circ \omega, \{(u,\phi(s)), \; (u,s)\in M\}).\]
If $\beta\circ\phi^{-1}$ and $\kappa\circ\phi^{-1}$ (the push-forwards of $\beta$ and $\kappa$ by $\phi$) still have no atoms, then $\phi(\mathbb{T})$ has the distribution of the genealogy of a $(\beta\circ\phi^{-1},\kappa\circ\phi^{-1})$ birth-death process, started at $\phi(t)\in J'$ and equipped with Poissonian mutations with rate $\mu\circ\phi^{-1}$.

Also, if $\kappa = 0$ and $\beta(J)=\infty$, then $\kappa\circ\phi^{-1}=0$ and $\beta\circ\phi^{-1}(J')=\infty$, and the measures $\mathscr{L}_{\mathbb{T}}$ and $\mathscr{L}_{\phi(\mathbb{T})}$ on $\partial\mathcal{T}$, defined for $\mathbb{T}$ and for $\phi(\mathbb{T})$, are the same.
\end{prop}

\begin{proof}
Suppose $\mathbb{T}$ is constructed as in Definition \ref{def_bd_process} with independent Poisson point processes $B_u$ and $D_u$ with respective intensities $\beta$ and $\kappa$, for each $u\in \bigcup_n\{0,1\}^n$.
This implies that the random sets defined by
\begin{gather*}
\phi(B_u) := \{\phi(s), \; s\in B_u\},\\
\phi(D_u) := \{\phi(s), \; s\in D_u\},
\end{gather*}
are independent Poisson point processes on the interval $J'$ with respective intensities $\beta \circ \phi^{-1}$ and $\kappa\circ\phi^{-1}$.
Remark that by assumption, for $\eta \in \{\beta,\kappa\}$, for all $t'\in J'$, we have $\eta\circ\phi^{-1} (\{t'\}) = 0$, so we a.s.\ have $t'\notin \phi(B_u)$ and $t'\notin\phi(D_u)$.
Now since $\alpha(u)$ is independent of $B_u$ and $D_u$, we have also a.s.
\begin{equation} \label{eq_alpha_notin_B_or_D}
\phi\circ\alpha(u) \notin \phi(B_u)\cup\phi(D_u).
\end{equation}

By definition, we have $\emptyset \in \mathcal{T}$ and $\alpha(\emptyset) = t$, so $\phi\circ\alpha(\emptyset) = \phi(t)$.
Then, if $u \in \mathcal{T}$, with $T_B(u)= \inf B_u\cap (\alpha(u),t_\infty)$, and $T_D(u)=\inf D_u\cap (\alpha(u),t_\infty)$, the following assertions hold.
\begin{itemize}
\item Since we have \eqref{eq_alpha_notin_B_or_D}, we know that a.s.\ for all $s\in B_u\cap(\alpha(u), t_\infty)$, we have $\phi(\alpha(u))< \phi(s)$.
This ensures that $\phi(T_B(u)) = \inf \phi(B_u)\cap (\phi\circ\alpha(u),t'_\infty)$.
\item For the same reason, we have $\phi(T_D(u)) = \inf \phi(D_u)\cap (\phi\circ\alpha(u),t'_\infty)$.
\item Because $\phi(B_u)$ is independent of $\phi(D_u)$ and because $\beta\circ\phi^{-1}$ and $\kappa\circ\phi^{-1}$ are diffuse by assumption, we have $\phi(B_u)\cap \phi(D_u) = \emptyset$ almost surely.
Therefore, we have:
\begin{itemize}
\item $\phi(T_B(u)) < \phi(T_D(u)) \iff T_B(u) < T_D(u)$, which implies $u0, u1 \in \mathcal{T}$, and $\phi\circ\alpha(u0) = \phi\circ\alpha(u1) = \phi\circ\omega(u) = \phi(T_B(u))$,
\item $\phi(T_D(u)) < \phi(T_B(u)) \iff T_D(u) < T_B(u) $, which implies $\phi\circ\omega(u) = \phi(T_D(u))$, and $u0, u1 \notin \mathcal{T}$,
\item $\phi(T_B(u)) = \phi(T_D(u)) = t'_\infty \iff T_B(u) = T_D(u) = t_\infty $, which implies $\phi\circ\omega(u) = t'_\infty$, and $u0, u1 \notin \mathcal{T}$.
\end{itemize}
\end{itemize}
Thus $(\mathcal{T}, \phi\circ\alpha, \phi\circ\omega)$ is defined as a $(\beta\circ\phi^{-1},\kappa\circ\phi^{-1})$ birth-death process, started at $\phi(t)$.

For the neutral mutations, we assume there is, as in Definition \ref{def_bd_process+mut}, a Poisson point process $\widetilde{M}$ on $(\bigcup_{n}\{0,1\}^n)\times J$ with intensity $\#\otimes\mu$, and such that:
\[M = \{(u,s) \in \widetilde{M},\; \,u\in \mathcal{T}, \alpha(u) < s \leq \omega(u)\}.\]
Now $\{(u,\phi(s)), \; (u,s)\in \widetilde{M}\}$ is a Poisson point process on $(\bigcup_{n}\{0,1\}^n)\times J'$ with intensity $\#\otimes\mu\circ\phi^{-1}$, so
\[ \{(u,\phi(s)), \; (u,s)\in M\} = \{(u,\phi(s)), \; (u,s) \in \widetilde{M},\; \,u\in \mathcal{T}, \alpha(u) < s \leq \omega(u)\} \]
is the definition of random neutral mutations at rate $\mu\circ\phi^{-1}$ on the tree $(\mathcal{T}, \phi\circ\alpha, \phi\circ\omega)$.

It remains to prove that in the case $\kappa = 0$ and $\beta(J) = \infty$, the measures $\mathscr{L}_{\mathbb{T}}$ and $\mathscr{L}_{\phi(\mathbb{T})}$ are the same.
By definition, we have for $u \in \bigcup_{n}\{0,1\}^n$,
\begin{align*}
\mathscr{L}_{\phi(\mathbb{T})}(B_u) &= \lim_{s'\uparrow t'_\infty} \frac{N'_u(s')}{\e^{\beta\circ\phi^{-1}([t'_0,s'])}} \\
&=  \lim_{s\uparrow t_\infty} \frac{N'_u(\phi(s))}{\e^{\beta\circ\phi^{-1}([t'_0,\phi(s)])}} ,
\end{align*}
where $N'_u(s') := \#\{v\in \mathcal{T}, \; u\preceq v, \, \phi\circ\alpha(v) < s \leq \phi\circ\omega(v)\}$ is the number of descendants of $u$ in the time-changed tree at time $s'$.
But we have a.s.\ for all $s\in J$, $N'_u(\phi(s))=N_u(s)$, and also $\beta\circ\phi^{-1}([t'_0,\phi(s)]) = \beta([t_0,s])$, so finally
\begin{align*}
\mathscr{L}_{\phi(\mathbb{T})}(B_u) &= \lim_{s\uparrow t_\infty} \frac{N'_u(\phi(s))}{\e^{\beta\circ\phi^{-1}([t'_0,\phi(s)])}} ,\\
&= \lim_{s\uparrow t_\infty} \frac{N_u(s)}{\e^{\beta([t_0,s])}}\\
&= \mathscr{L}_{\mathbb{T}}(B_u) ,
\end{align*}
which ends the proof.
\end{proof}

\begin{prop}[Characterization of pure-birth processes] \label{prop_bd_process_charac}
Let $J = [t_0, t_\infty)$ be a real interval, with $-\infty < t_0 <t_\infty\leq\infty$, and $\beta$ a diffuse Radon measure on $J$, such that $\beta(J) = \infty$.

There is a unique family $(\PP_t)_{t\in J}$ of distributions on simple trees $(\mathcal{T}, \alpha, \omega, \mathscr{L})$ equipped with a measure $\mathscr{L}$ on $\partial \mathcal{T} := \{0,1\}^{\N}$, such that for all $t\in J$
\begin{enumerate}[(i)]
\item  $\mathcal{T} = \bigcup_n\{0,1\}^n$ and $\alpha(\emptyset) = t\quad\PP_t$-almost surely.
\item  $\PP_t(\omega(\emptyset) > s) = \e^{-\beta([t,s))}$.
\item Under $\PP_t$, $\mathscr{L}(\partial \mathcal{T})$ is an exponential r.v.\ with mean $\e^{-\beta([t_0,t))}$.
\item Under $\PP_t$, define for $i\in\{0,1\}$,
$\alpha_i(u) := \alpha(iu)$, $\omega_i(u) := \omega(iu)$, $\mathscr{L}_i$ the measure on $\partial\mathcal{T}$ such that $\mathscr{L}_i(B_u) = \mathscr{L}(B_{iu})$ for all $u\in\mathcal{T}$ and finally $\mathbb{T}_i := (\mathcal{T}, \alpha_i, \omega_i, \mathscr{L}_i)$.
Then the conditional distribution of the pair of trees $(\mathbb{T}_0,\mathbb{T}_1)$ given $\omega(\emptyset)$ is $\PP_{\omega(\emptyset)}^{\otimes 2}$, {i.e.} they are independent with the same distribution $\PP_{\omega(\emptyset)}$.
\end{enumerate}

Furthermore, for all $t\in J$, $\PP_t$ is the distribution of the genealogy of a pure-birth process with birth rate $\beta$ started with one individual at time $t\in J$, equipped with $\mathscr{L}$ the measure on $\partial \mathcal{T}$ introduced in Definition \ref{def_bd_proc_measure_on_boundary}.
\end{prop}

\begin{proof}
Let $\mathbb{Q}_t$ be the law of the genealogy of a $\beta$ pure-birth process started from $t$.
We will first show that the family $(\mathbb{Q}_t)_{t\in J}$ satisfies the assertions \emph{(i)-(iv)} of the theorem.

\emph{(i)} By definition $\alpha(\emptyset) = t$.
Also, the fact that for all $t\in T$, $\beta([t,t_\infty))=\infty$, implies that for each Poisson point process with intensity $\beta$ on $J$, there are infinitely many points in $[t, t_\infty)$.
This implies that each individual in the process will eventually split into two, so that $\mathcal{T} = \bigcup_n\{0,1\}^n\quad\PP_t$-almost surely.

\emph{(ii)} Under $\mathbb{Q}_t$, $\omega(\emptyset)$ is distributed as the first point of a Poisson point process $B_\emptyset$ on $[t, t_\infty)$ with intensity $\beta$.
Therefore,
\[\mathbb{Q}_t(\omega(\emptyset)> s) = \mathbb{Q}_t(\# B_\emptyset \cap [t, s) = 0) =\e^{-\beta([t,s))}. \]

\emph{(iii)} By Proposition \ref{prop_bd_process}, writing $\E_t$ for the expectation under $\mathbb{Q}_t$, we have for $t<T<t_\infty$,
\[ \E_t(\e^{-\alpha N_T}) = 1-\frac{(1-\e^{-\alpha})}{\e^{- \beta([t,T])}+(1-\e^{-\alpha})(1-\e^{- \beta([t,T])})}. \]
Replacing $\alpha$ by $\alpha \e^{-\beta([t_0, T])}$ and letting $T\to t_\infty$, we have by dominated convergence:
\[\E_t(\e^{-\alpha \mathscr{L}(\partial \mathcal{T})}) = \frac{1}{\alpha\e^{-\beta([t_0, t))} + 1,} \]
which implies that $\mathscr{L}(\partial\mathcal{T})$ is an exponential random variable with mean $\e^{-\beta([t_0, t))}$.

\emph{(iv)} Let us define a family $(B_u)_{u\in\mathcal{T}}$ of independent Poisson point processes on $J$ with intensity $\beta$.
Let us write $F$ for the deterministic function such that for all $t\in J$, $F(t, (B_u)_{u\in\mathcal{T}})$ is the simple tree $\mathbb{T}=(\mathcal{T}, \alpha, \omega, \mathscr{L})$ constructed as in Definition \ref{def_bd_process}, which follows the distribution $\mathbb{Q}_t$.
By assumption, the two families $(B_{0u})_{u\in\mathcal{T}}$ and $(B_{1u})_{u\in\mathcal{T}}$ are independent, and by construction, we have
\[\mathbb{T}_0 = F(\omega(\emptyset), (B_{0u})_{u\in\mathcal{T}}) \text{ and }\mathbb{T}_1 = F(\omega(\emptyset), (B_{1u})_{u\in\mathcal{T}}), \]
where $\mathbb{T}_0$ and $\mathbb{T}_1$ are defined as in the statement of the Proposition.
Therefore, under $\mathbb{Q}_t$, the conditional distribution of $(\mathbb{T}_0, \mathbb{T}_1)$ given $\omega(\emptyset)$ is $\PP_{\omega(\emptyset)}^{\otimes 2}$.

Now, let us show that if a family $(\PP_t)_{t\in J}$ satisfies the assertions \textit{(i)-(iv)} of the Proposition, it satisfies also the following one.
Let $\mathcal{T}_n$ be the complete binary tree with $n$ generations
\[\mathcal{T}_n := \bigcup_{k=0}^n \{0,1\}^k,\]
and let $\PP_t^n$ be the distribution of $(\alpha(u), \omega(u), \mathscr{L}(B_u))_{u\in\mathcal{T}_n}$, where $(\mathcal{T}, \alpha, \omega, \mathscr{L})$ has distribution $\PP_t$.
Now we view $\PP_t^n$ as a probability measure on the space \\$(\R^3)^{\mathcal{T}_n} = \{(x(u), y(u),z(u)), u\in\mathcal{T}_n\}$.
Then we have
\begin{enumerate}
\item $x(\emptyset) := t\quad\PP^n_t$-almost surely.
\item For all $m\leq n$ and $u\in\mathcal{T}_m$, conditional on $x(u)$ and independently of the variables $(x(v), y(v))_{v\in\mathcal{T}_m\setminus\{u\}}$, the distribution of $y(u)$ is given by:
\[  \PP^n_t(y(u) > s) = \e^{-\beta([x(u), s))}\qquad s\geq x(u). \]
\item For all $u\in\{0,1\}^n$, conditional on $x(u)$ and independently of the rest, $z(u)$ is defined as an exponential random variable with mean $\e^{-\beta([t_0,x(u)))}$.
\item For all $u\in\mathcal{T}_{n-1}$, $x(u0) = x(u1) := y(u)$.
\item For all $u\in\mathcal{T}_{n-1}$, $z(u) := z(u0)+z(u1)$.
\end{enumerate}
Indeed, assertion 1 is directly deduced from \textit{(i)}, 5 is trivial because $\mathscr{L}$ is additive, and 2, 3 and 4 are proved by induction on $n$ using \textit{(iv)}.
One can check that 2 stems from \textit{(ii)} and \textit{(iv)},
3 from \textit{(iii)} and \textit{(iv)}, and 4 from \textit{(i)} and \textit{(iv)}.

Now it is clear that these five assumptions define $\mathbb{P}^n_t$ uniquely for $n\geq 0$ and $t\in J$.
Also, a measured simple tree $(\mathcal{T}, \alpha, \omega, \mathscr{L})$ for which $\mathcal{T} = \bigcup_n\{0,1\}^n$ is entirely described by $(\alpha(u), \omega(u), \mathscr{L}(B_u))_{u\in\mathcal{T}} \in (\R^3)^{\mathcal{T}}$.
This implies that $\mathbb{P}_t$ is uniquely determined by its marginal distribution $(\PP_t^n)_{n\geq 0}$.

Finally, we have shown that the family $(\mathbb{Q}_t)_{t\in J}$, where $\mathbb{Q}_t$ is the law of the genealogy of a $\beta$ pure-birth process started from $t$, satisfies assertions \textit{(i)-(iv)}.
In addition, we have shown that there is at most one family $(\PP_t)$ of simple tree distributions satisfying assertions \textit{(i)-(iv)}.
Therefore, such a family exists and is unique, which concludes the proof.

\end{proof}

\subsection{Proof of Lemmas \ref{lemma_link_bd_cpp} and \ref{lemma_link_bd_cpp+meas}}
\label{subsec:proof_lemma_link_bd_cpp}

  Let us write $\PP_z$ for the distribution of a CPP$(\nu, z)$.
  Let $\mathcal{N}$ be a Poisson point process with intensity $\dd t\otimes\nu$ as in our construction of CPP trees.
Recall that $T(z) = \inf\{t \geq 0, (x,t) \in \mathcal{N}, x \geq z\}$ and define  \[ \mathcal{N}_{z} := \mathcal{N}\cap ([0, T(z))\times[0,z]).\]
  Define also $\mathbb{T}^z$ as the comb function tree given by $\mathcal{N}_z$ with distribution denoted $\PP_z$.
  Write $\mathcal{P}_{z}$ for the distribution of the pair $(\mathcal{N}_{z}, T(z))$.
  
  In Proposition \ref{prop_bd_process_charac}, we characterized the distributions of pure-birth processes.
  As a result, to conclude the present proof, it is sufficient to show that the family $(\PP_z)_{z\in J}$ satisfies the following conditions:
  \begin{enumerate}[(i)]
    \item We have $\mathcal{T} = \bigcup_n\{0,1\}^n$ and $\alpha(\emptyset) = z\quad\PP_z$-almost surely.
    \item We have $\PP_z(\omega(\emptyset) < x) = \e^{-\beta((x,z])}$.
    \item Under $\PP_z$, $\mathscr{L}(\partial \mathcal{T})$ is an exponential r.v.\ with mean $\e^{-\beta((z,z_0])}$.
    \item Under $\PP_z$, define for $i\in\{0,1\}$,
    $\alpha_i(u) := \alpha(iu)$, $\omega_i(u) := \omega(iu)$, $\mathscr{L}_i$ the measure on $\partial\mathcal{T}$ such that $\mathscr{L}_i(B_u) = \mathscr{L}(B_{iu})$ for all $u\in\mathcal{T}$ and finally $\mathbb{T}_i := (\mathcal{T}, \alpha_i, \omega_i, \mathscr{L}_i)$.
    Then the conditional distribution of the pair of trees $(\mathbb{T}_0,\mathbb{T}_1)$ given $\omega(\emptyset)$ is $\PP_{\omega(\emptyset)}^{\otimes 2}$, {i.e.} they are independent with the same distribution $\PP_{\omega(\emptyset)}$.
  \end{enumerate}
  Let us now prove each assertion.
  
  \textit{(i)} Since $\nu(\Rplus) = \infty$ we have a.s.\ for any $0\leq a < b \leq T(z)$:
  \[\#(\mathcal{N}_z\cap [a,b]\times\Rplus)=\infty.\]
  Also, since $\nu$ is diffuse, we have a.s.\ for all $x>0$ that $\#(\mathcal{N}\cap \Rplus \times \{x\}) \leq 1$
  Those two conditions imply that $\mathbb{T}^z$ is a complete binary tree.
  
  \textit{(ii) -- (iii)} The first branching point of the tree $\mathbb{T}^z$ is $\omega(\emptyset) =\max\{x>0, \; (t,x) \in\mathcal{N}_z\}.$
  Also the total mass of the tree is $\mathscr{L}(\partial\mathcal{T})=T(z)$, which is an exponential random variable with mean $(\overline{\nu}(z))^{-1}=\e^{-\beta((z,z_0])}$.
  We can easily compute the distribution of $\omega(\emptyset)$ under $\mathcal{P}_z$, since conditional on $T(z)$, $\mathcal{N}_z$ is a Poisson point process on $[0,T(z))\times [0,z]$ with intensity $\dd t\otimes\nu$.
  Therefore, for $x\in (0,z]$:
  \begin{align*}
  \mathcal{P}_z(\omega(\emptyset) < x) &= \int_0^\infty \PP(T(z) \in \dd t) \e^{-t \nu([x, z])} \\
  &= \int_0^\infty \overline{\nu}(z) \e^{-\overline{\nu}(z)t} \e^{-t(\overline{\nu}(x)-\overline{\nu}(z))} \dd t\\
  &= \int_0^\infty \overline{\nu}(z) \e^{-\overline{\nu}(x)t} \dd t\\
  &= \frac{\overline{\nu}(z)}{\overline{\nu}(x)} = \e^{-\beta((x,z])}.
  \end{align*}
  
  \textit{(iv)} It remains to prove the branching property for the family $(\PP_z)_{z \in (0,z_0]}$.
  
  Under $\mathcal{P}_z$, conditional on $\omega(\emptyset)$, let $(\mathcal{N}_1, T_1)$ and $(\mathcal{N}_2, T_2)$ be independent random variables of identical distribution $\mathcal{P}_{\omega(\emptyset)}$.
  We concatenate $\mathcal{N}_1$ and $\mathcal{N}_2$, adding a point of height $\omega(\emptyset)$ between the two sets:
  \[\widetilde{\mathcal{N}} = \mathcal{N}_1 \cup \{(T_1, \omega(\emptyset))\} \cup \{(T_1 + t, x), (t,x) \in \mathcal{N}_2 \}. \]
  We claim that the following equality in distribution holds:
  \begin{equation} \label{eq_poisson_branch}
  (\widetilde{\mathcal{N}},T_1 + T_2) \eqdist (\mathcal{N}_{z}, T(z)),
  \end{equation}
  which formulates the branching property for the family $(\PP_z)_{z\in (0,z_0]}$.
  
  From basic properties of Poisson point processes, we know that conditional on $T(z)$, the highest atom of $\mathcal{N}_z$ is $(U, Z)$, with $U$ having a uniform distribution on $[0, T(z)]$ and $Z:=\omega(\emptyset)$ independent of $U$, such that
  \[\mathcal{P}^z(Z \leq x \mid T(z)) = \e^{-T(z)(\overline{\nu}(x)-\overline{\nu}(z))}.\]
  The joint distribution of $(Z, T(z))$ is therefore given by:
  \begin{align*}
  \E[f(T(z))\mathbbm{1}_{Z \leq x}] &= \int_0^\infty \overline{\nu}(z) \e^{-\overline{\nu}(z)t} \e^{-t(\overline{\nu}(x)-\overline{\nu}(z))} f(t) \dd t\\
  &= \int_0^\infty \overline{\nu}(z) \e^{-\overline{\nu}(x) t} f(t) \dd t\\
  &= \int_0^\infty \overline{\nu}(z) \int_{\overline{\nu}(x)}^\infty t \e^{- u t} \dd u \; f(t) \dd t\\
  &= \int_{\overline{\nu}(x)}^\infty \frac{\overline{\nu}(z)}{u^2} \int_0^\infty t u^2 \e^{- u t} f(t) \dd t \; \dd u \\
  \end{align*}
  In other words, the random variable $\overline{\nu}(Z)$ has a density $\frac{\overline{\nu}(z)}{u^2} \1_{u \geq \overline{\nu}(z)} \dd u$, and conditional on $\overline{\nu}(Z)$, $T(z)$ follows a Gamma distribution with parameter $(\overline{\nu}(Z), 2)$.
  As $U/{T(z)}$ is uniform on $[0, 1]$ and independent of $Z$, one can check that $(Z, T(z), U)$ has the same distribution as $(Z, T_1 + T_2, T_1)$, where conditional on $Z$, the variables $T_1$ and $T_2$ are independent with the same exponential distribution with parameter $\overline{\nu}(Z)$.
 This concludes the proof of \eqref{eq_poisson_branch} since conditional on $(Z, T(z), U)$ ({resp.} $(Z, T_1 + T_2, T_1)$), $\mathcal{N}_z\setminus\{(U,Z)\}$ ({resp.} $\widetilde{\mathcal{N}}\setminus\{(T_1, Z)\}$) is a Poisson point process on $[0,T(z))\times [0,Z]$ ({resp.} on $[0,T_1+T_2)\times [0,Z]$) with intensity $\dd t\otimes \nu$.

\subsection{Subordinators and Regenerative Sets} \label{app_sub_reg}
We use some classical results about regenerative sets and subordinators, whose proofs can be found in the first two sections of Bertoin's Saint-Flour lecture notes \citep{bertoin_subordinators:_1997}.

\begin{definition}
A \textbf{subordinator} is a right-continuous, increasing Markov process $(\sigma_t)_{t\geq 0}$ started from $0$ with values in $[0, \infty]$, where $\infty$ is an absorbing state, such that for all $s<t$, conditional on $\{\sigma_s < \infty\}$, we have
\[\sigma_t - \sigma_s \eqdist \sigma_{t-s}. \]
\end{definition}

\begin{theorem}
The distribution of a subordinator is characterized by its \textbf{Laplace exponent} defined as the increasing function $\phi: \Rplus \to \Rplus$, such that for all $\lambda, t \geq 0$,
\[\E[\e^{-\lambda \sigma_t}] = \e^{-t\phi(\lambda)}, \]
with the convention $\e^{-\lambda \infty} = 0$ for all $\lambda \geq 0$.
The Laplace exponent can be written under the form
\[\phi(\lambda) = k + d\lambda + \int_{(0, \infty)} (1 - \e^{-\lambda x}) \pi(\dd x),\]
where $k$ is called the \textbf{killing rate}, $d$ the \textbf{drift coefficient} and $\pi$ the \textbf{Lévy measure} of the subordinator.
Necessarily, we have $k,d \geq 0$ and $\pi$ satisfies 
\[\int_{(0, \infty)} (1\wedge x) \pi(\dd x) < \infty.\]
Letting $\zeta := \inf\{t \geq 0, \; \sigma_t = \infty\}$ be the lifetime of the subordinator, $\zeta$ follows an exponential distribution with parameter $k$ (if $k =0$, then $\zeta \equiv \infty$).
Also we have almost surely for all $t < \zeta$,
\[\sigma_t = dt + \sum_{s\leq t} \Delta\sigma_s,\]
and the set of jumps $\{(s, \Delta \sigma_s), \; \Delta\sigma_s > 0\}$ is a Poisson point process with intensity $\dd s \otimes \pi$.
\end{theorem}

The \textbf{renewal measure} of a subordinator is defined as the measure $U(\dd x)$ on $\Rplus$ such that for any non-negative measurable function $f$
\[\int_{\Rplus} f(x) U(\dd x) = \E \left [\int_0^\zeta f(\sigma_t) \dd t\right ]. \]
This renewal measure characterizes the distribution of $\sigma$ since its Laplace transform is the inverse of $\phi$
\[\frac{1}{\phi(\lambda)} = \int_{\Rplus} \e^{-\lambda x} U(\dd x).\]
Remark also that setting $L_x := \inf \{t \geq 0, \sigma_t > x\}$ the right-continuous inverse of $\sigma$, we have
\[U(x) := U([0,x]) = \E \left [\int_0^\infty \1_{\sigma_t \leq x} \, \dd t \right ] = \E[L_x].\]

\begin{definition}
Given a probability space $(\Omega, \mathcal{F}, \PP)$ equipped with a complete, right-continuous filtration $(\mathcal{F}_t)_{t\geq 0}$, a \textbf{regenerative set} $R$ is a random closed set containing $0$ for which the following properties hold
\begin{itemize}
	
	\item \emph{Progressive measurability.} For all $t\geq 0$, the set
	$\{ (s, \omega)  \in [0,t]\times \Omega , \; s \in R(\omega) \} $
	is in $\mathcal{B}([0,t])\otimes\mathcal{F}_t$.
	
	\item \emph{Regeneration property.} For a $(\mathcal F_t)_{t\geq 0}$-stopping time $T$ such that a.s.\ on $\{T< \infty\}$, $T\in R$ and $T$ is not right-isolated in $R$, we have:
	\[ R\cap[T, \infty[ - T \eqdist R,\]
	where $R\cap[T, \infty[ - T$ is defined formally as the set $\{ t\geq 0,\; T+t \in R \}$.
\end{itemize}
\end{definition}

We define the range of a subordinator $\sigma$ as the closed set $\overline{\{\sigma_t, \; t \geq 0\}}$, and see that all regenerative sets can be expressed in this form.

\begin{theorem}
The range of a subordinator is a regenerative set. Conversely, if $R$ is a regenerative set without isolated points, there exists a subordinator $\sigma$ whose range is $R$ almost surely.
\end{theorem}

\begin{remark} \label{rmq_heavy_regen_set_renewal_measure}
  In the case where $\lambda (R)>0$ a.s., one can define such a subordinator as
  \[\sigma_t := \inf\{x \geq 0, \; \lambda([0,x]\cap R) > t \}. \]
  Then $\sigma$ is the unique subordinator with drift $1$ and range $R$, and its renewal measure is $U(\dd x) = \PP(x \in R) \,\dd x$.
  Notice that $\lambda(R) = \inf \{t \geq 0, \; \sigma_t = \infty\} = \zeta$ by definition.
  Therefore $\lambda(R)$ is an exponential random variable with parameter $k$, the killing rate of $\sigma$.
\end{remark}

\bibliographystyle{abbrvnat}
\bibliography{refs}

\addcontentsline{toc}{section}{References}
\phantomsection

\end{document}